\documentclass[12pt]{amsart}
\usepackage{amsmath,amssymb,latexsym,cancel,rotating}
\usepackage{graphicx,amssymb,mathrsfs,amsmath,color,fancyhdr,amsthm}
\usepackage[all]{xy}

\newtheorem{theorem}{Theorem}[section]
\newtheorem{lemma}[theorem]{Lemma}
\newtheorem{corollary}[theorem]{Corollary}

\newtheorem{proposition}[theorem]{Proposition}

\begin{document}

\title[Geometric realizations of Lusztig's symmetries]
{Geometric realizations of Lusztig's symmetries of symmetrizable quantum groups}

\author[Zhao]{Minghui Zhao}
\address{College of Science, Beijing Forestry University, Beijing 100083, P. R. China}
\email{zhaomh@bjfu.edu.cn}


\date{\today}


\bibliographystyle{abbrv}

\maketitle

\begin{abstract}
The geometric realizations of Lusztig's symmetries of symmetrizable quantum groups are given in this paper.
This construction is a generalization of that in \cite{Xiao_Zhao_Geometric_realizations_of_Lusztig's_symmetries}.
\end{abstract}

\tableofcontents

\section{Introduction}


Let $\mathbf{U}$ be the quantum group and $\mathbf{f}$ be the Lusztig's algebra associated with a  symmetrizable generalized Cartan matrix.
There are two well-defined $\mathbb{Q}(v)$-algebra embeddings ${^+}:\mathbf{f}\rightarrow\mathbf{U}$ and ${^-}:\mathbf{f}\rightarrow\mathbf{U}$ with images $\mathbf{U}^+$ and $\mathbf{U}^-$, where $\mathbf{U}^+$ and $\mathbf{U}^-$ are the positive part and the negative part of $\mathbf{U}$ respectively.

When the Cartan matrix is symmetric, Lusztig introduced the geometric realization of $\mathbf{f}$ and the canonical basis of it in \cite{Lusztig_Canonical_bases_arising_from_quantized_enveloping_algebra,Lusztig_Quivers_perverse_sheaves_and_the_quantized_enveloping_algebras}.
In \cite{Lusztig_Introduction_to_quantum_groups}, Lusztig generalized the geometric realization to $\mathbf{f}$ associated with a symmetrizable generalized Cartan matrix.

Let $\tilde{Q}=(Q,a)$ be a quiver with automorphism corresponding to $\mathbf{f}$, where $Q=(\mathbf{I},H)$. Let $\mathbf{V}$ be an $\mathbf{I}$-graded vector space such that $a\mathbf{V}=\mathbf{V}$ and $\underline{\dim}\mathbf{V}=\nu\in\mathbb{N}\mathbf{I}^a$.
Consider the variety $E_{\mathbf{V}}$ consisting of representations of $Q$ with dimension vector $\nu$ and a category $\mathcal{Q}_{\mathbf{V}}$
of some semisimple perverse sheaves (\cite{Beilinson_Bernstein_Deligne_Faisceaux_pervers,Bernstein_Lunts_Equivariant_sheaves_and_functors,Kiehl_Weissauer_Weil_conjectures_perverse_sheaves_and_l'adic_Fourier_transform}) on $E_{\mathbf{V}}$.

The isomorphism $a:\mathbf{V}\rightarrow\mathbf{V}$ induces
a functor $a^\ast:\mathcal{Q}_{\mathbf{V}}\rightarrow\mathcal{Q}_{\mathbf{V}}$. Lusztig defined a new category $\tilde{\mathcal{Q}}_{\mathbf{V}}$ consisting of objects $(\mathcal{L},\phi)$, where $\mathcal{L}$ is an object in $\mathcal{Q}_{\mathbf{V}}$ and $\phi:a^\ast\mathcal{L}\rightarrow\mathcal{L}$ is an isomorphism.
Lusztig considered a submodule $\mathbf{k}_{\nu}$ of the Grothendieck group $K(\tilde{\mathcal{Q}}_{\mathbf{V}})$. Considering all dimension vectors, he proved that $\mathbf{k}=\bigoplus_{\nu\in\mathbb{N}I}\mathbf{k}_{\nu}$ is isomorphic to $\mathbf{f}$.

Lusztig also introduced some symmetries $T_i$ on $\mathbf{U}$ for all $i\in I$ in \cite{Lusztig_Quantum_deformations_of_certain_simple_modules_over_enveloping_algebras,Lusztig_Quantum_groups_at_roots_of_1}.
Since $T_i(\mathbf{U}^+)$ is not contained in $\mathbf{U}^+$, Lusztig introduced two subalgebras ${_i\mathbf{f}}$ and ${^i\mathbf{f}}$ of $\mathbf{f}$ for any $i\in I$, where ${_i\mathbf{f}}=\{x\in\mathbf{f}\,\,|\,\,T_i(x^+)\in\mathbf{U}^+\}$ and ${^i\mathbf{f}}=\{x\in\mathbf{f}\,\,|\,\,T^{-1}_i(x^+)\in\mathbf{U}^+\}$.
Let $T_i:{_i\mathbf{f}}\rightarrow{^i\mathbf{f}}$ be the unique map satisfying
$T_i(x^+)=T_i(x)^+$.
For any $i\in I$, ${_i\mathbf{f}}$ and ${^i\mathbf{f}}$ are the subalgebras of $\mathbf{f}$ generated by $f(i,j;m)$ and $f'(i,j;m)$ for all $i\neq j\in I$ and $-a_{ij}\geq m\in\mathbb{N}$ respectively.

Associated to a finite dimensional hereditary algebra, Ringel introduced the Hall algebra and its composition subalgebra in \cite{Ringel_Hall_algebras_and_quantum_groups},
which gives a realization of $\mathbf{U}^+$.
Via the Hall algebra approach, one can apply BGP-reflection functors to quantum groups to give precise constructions of Lusztig's symmetries (\cite{
Ringel_PBW-bases_of_quantum_groups,Lusztig_Canonical_bases_and_Hall_algebras,
Sevenhant_Van_den_Bergh_On_the_double_of_the_Hall_algebra_of_a_quiver,
Xiao_Yang_BGP-reflection_functors_and_Lusztig's_symmetries,Deng_Xiao,
Xiao_Zhao_BGP-reflection_functors_and_Lusztig's_symmetries_of_modified_quantized_enveloping_algebras}).

Let $i\in I$ be a sink (resp. source) of $Q$. Similarly to the geometric realization of $\mathbf{f}$, consider a subvariety
${_i{E_{\mathbf{V}}}}$ (resp. ${^i{E_{\mathbf{V}}}}$) of $E_{\mathbf{V}}$ and a category ${_i\mathcal{Q}}_{\mathbf{V}}$ (resp. ${^i\mathcal{Q}}_{\mathbf{V}}$) of some semisimple perverse sheaves on ${_i{E_{\mathbf{V}}}}$ (resp. ${^i{E_{\mathbf{V}}}}$). In \cite{Xiao_Zhao_Geometric_realizations_of_Lusztig's_symmetries}, it was showed that $\oplus_{\nu\in\mathbb{N}I}K({_i\mathcal{Q}}_{\mathbf{V}})$ (resp. $\oplus_{\nu\in\mathbb{N}I}K({^i\mathcal{Q}}_{\mathbf{V}})$) realizes
${_i\mathbf{f}}$ (resp. ${^i\mathbf{f}}$).

Let $i\in I$ be a sink of $Q$ and $Q'=\sigma_iQ$ be the quiver by reversing the directions
of all arrows in $Q$ containing $i$. Hence, $i$ is a source of $Q'$. Consider two $I$-graded vector spaces $\mathbf{V}$ and $\mathbf{V}'$ such that $\underline{\dim}\mathbf{V}'=s_i(\underline{\dim}\mathbf{V})$.
In the case of finite type, Kato introduced an equivalence $\tilde{\omega}_i:{_i\mathcal{Q}}_{\mathbf{V},Q}\rightarrow{^i\mathcal{Q}_{\mathbf{V}',Q'}}$ and studied the properties of this equivalence in \cite{Kato_An_algebraic_study_of_extension_algebra,Kato_PBW_bases_and_KLR_algebras}. In \cite{Xiao_Zhao_Geometric_realizations_of_Lusztig's_symmetries}, his construction was generalized to all cases. It was proved that the map induced by $\tilde{\omega}_i$ realizes the Lusztig's symmetry $T_i:{_i\mathbf{f}}\rightarrow{^i\mathbf{f}}$ by using the relations between $\tilde{\omega}_i$ and the Hall algebra approach to $T_i$ in \cite{Lusztig_Canonical_bases_and_Hall_algebras}.

In \cite{Lusztig_Braid_group_action_and_canonical_bases}, Lusztig showed that Lusztig's symmetries and canonical bases are compatible.
The main result in \cite{Xiao_Zhao_Geometric_realizations_of_Lusztig's_symmetries} gives a geometric interpretation of Lusztig's result in \cite{Lusztig_Braid_group_action_and_canonical_bases}.

In this paper, we shall generalize the construction in \cite{Xiao_Zhao_Geometric_realizations_of_Lusztig's_symmetries} and give geometric realizations of Lusztig's symmetries of symmetrizable quantum groups.

Let $\tilde{Q}=(Q,a)$ be a quiver with automorphism. Fix $i\in I=\mathbf{I}^a$ and assume that $\mathbf{i}$ is a sink (resp. source) for any $\mathbf{i}\in i$.
Similarly to the category $\tilde{\mathcal{Q}}_{\mathbf{V}}$, we can define
${_i\tilde{\mathcal{Q}}}_{\mathbf{V}}$ (resp. ${^i\tilde{\mathcal{Q}}}_{\mathbf{V}}$).
Consider a submodule ${_i\mathbf{k}}_{\nu}$ (resp. ${^i\mathbf{k}}_{\nu}$) of $K({_i\tilde{\mathcal{Q}}}_{\mathbf{V}})$ (resp. $K({^i\tilde{\mathcal{Q}}}_{\mathbf{V}})$).
We verify that $\oplus_{\nu\in\mathbb{N}I}{_i\mathbf{k}}_{\nu}$ (resp. $\oplus_{\nu\in\mathbb{N}I}{^i\mathbf{k}}_{\nu}$) realizes
${_i\mathbf{f}}$ (resp. ${^i\mathbf{f}}$) by using the result in \cite{Xiao_Zhao_Geometric_realizations_of_Lusztig's_symmetries} and the relation between $\tilde{\mathcal{Q}}_{\mathbf{V}}$ and $\mathcal{Q}_{\mathbf{V}}$.

Let $i\in I=\mathbf{I}^a$ such that $\mathbf{i}$ is a sink, for any $\mathbf{i}\in\mathfrak{i}$.
Let $Q'=\sigma_i Q$ be the quiver by reversing the directions
of all arrows in $Q$ containing $\mathbf{i}\in i$.
So for any $\mathbf{i}\in i$, $\mathbf{i}$ is a source of $Q'$.

Consider two $\mathbf{I}$-graded vector spaces $\mathbf{V}$ and $\mathbf{V}'$ such that $a\mathbf{V}=\mathbf{V}$, $a\mathbf{V}'=\mathbf{V}'$  and $\underline{\dim}\mathbf{V}'=s_i(\underline{\dim}\mathbf{V})$.
In this paper, it is proved that the equivalence $\tilde{\omega}_i:{_i\mathcal{Q}}_{\mathbf{V},Q}\rightarrow{^i\mathcal{Q}_{\mathbf{V}',Q'}}$ is compatible with $a^\ast$. Hence we get a functor $\tilde{\omega}_i:{_i\tilde{\mathcal{Q}}}_{\mathbf{V},Q}\rightarrow{^i\tilde{\mathcal{Q}}_{\mathbf{V}',Q'}}$ and a map
$\tilde{\omega}_i:{_i\mathbf{k}}\rightarrow{{^i\mathbf{k}}}$.
We also prove $\tilde{\omega}_i:{_i\mathbf{k}}\rightarrow{{^i\mathbf{k}}}$ is an isomorphism of algebras.


Assume that $\underline{\dim}\mathbf{V}=m\gamma_i+\gamma_j$, where $\gamma_i=\sum_{\mathbf{i}\in i}\mathbf{i}$ and $\gamma_j=\sum_{\mathbf{i}\in j}\mathbf{i}$.
We construct a series of distinguished triangles in $\mathcal{D}_{G_{\mathbf{V},Q}}(E_{\mathbf{V},Q})$,
which represent the constant sheaf $\mathbf{1}_{_iE_{\mathbf{V},Q}}$ in terms of some semisimple perverse sheaves $I_p\in\mathcal{D}_{G_{\mathbf{V},Q}}(E_{\mathbf{V},Q})$ geometrically.
Applying to the Grothendieck group, $\mathbf{1}_{_iE_{\mathbf{V},Q}}$ corresponds to $f(i,j;m)$.
Assume that $\underline{\dim}\mathbf{V}'=s_i(\underline{\dim}\mathbf{V})=m'\gamma_i+\gamma_j$.
Applying to the Grothendieck group, $\mathbf{1}_{^iE_{\mathbf{V}',Q'}}$ corresponds to $f'(i,j;m')$ similarly.
The properties of BGP-reflection functors imply $\tilde{\omega}_i(v^{-mN}\mathbf{1}_{_iE_{\mathbf{V},Q}})=v^{-m'N}\mathbf{1}_{^iE_{\mathbf{V}',Q'}}.$
Since ${_i\mathbf{f}}$ (resp. ${^i\mathbf{f}}$) is generated by $f(i,j;m)$ (resp. $f'(i,j;m)$), we have the following commutative diagram
$$\xymatrix{
{_i\mathbf{k}}\ar[r]^-{\tilde{\omega}_i}\ar[d]^-{}&{{^i\mathbf{k}}}\ar[d]^-{}\\
{_i}\mathbf{f}_{\mathcal{A}}\ar[r]^-{T_i}&{^i\mathbf{f}}_\mathcal{A}
}$$
That is, $\tilde{\omega}_i$ gives a geometric realization of Lusztig's symmetry $T_i$ for any $i\in I$.

\section{Quantum groups and Lusztig's symmetries}\label{section:2}

\subsection{Quantum groups}\label{subsection:2.1}

Fix a finite index set $I$ with $|I|=n$. Let $A=(a_{ij})_{i,j\in I}$ be a symmetrizable generalized Cartan matrix and $D=\textrm{diag}(\varepsilon_i\,\,|\,\,i\in I)$ be a diagonal matrix such that $DA$ is symmetric.
Let $(A,\Pi,\Pi^{\vee},P,P^{\vee})$ be a Cartan datum associated with
$A$, where
\begin{enumerate}
  \item[(1)]$\Pi=\{\alpha_i\,\,|\,\,i\in I\}$ is the set of simple roots;
  \item[(2)]$\Pi^{\vee}=\{h_i\,\,|\,\,i\in I\}$ is the set of simple coroots;
  \item[(3)]$P$ is the weight lattice;
  \item[(4)]$P^{\vee}$ is the dual weight lattice.
\end{enumerate}
Let $\mathfrak{h}=\mathbb{Q}\otimes_{\mathbb{Z}}P^{\vee}$ and there exist a symmetric bilinear form $(-,-)$ on $\mathfrak{h}^{\ast}$ such that $(\alpha_i,\alpha_j)=\varepsilon_ia_{ij}$ for any $i,j\in I$ and $\lambda(h_i)=2\frac{(\alpha_i,\lambda)}{(\alpha_i,\alpha_i)}$ for any $\lambda\in\mathfrak{h}^{\ast}$ and $i\in I$.

Fix an indeterminate $v$. Let $v_i=v^{\varepsilon_i}$. For any $n\in\mathbb{Z}$, set
$$[n]_{v_i}=\frac{v_i^n-v_i^{-n}}{v_i-v_i^{-1}}\in\mathbb{Q}(v).$$
Let $[0]_{v_i}!=1$ and $[n]_{v_i}!=[n]_{v_i}[n-1]_{v_i}\cdots[1]_{v_i}$ for any $n\in\mathbb{Z}_{>0}$.

Let $\mathbf{U}$ be the quantum group corresponding to $(A,\Pi,\Pi^{\vee},P,P^{\vee})$ generated by the elements $E_i, F_i (i\in I)$ and $K_{\mu} (\mu\in P^{\vee})$. Let $\mathbf{U}^+$ (resp. $\mathbf{U}^-$) be the positive (resp. negative) part of $\mathbf{U}$ generated by $E_i$ (resp. $F_i$) for all $i\in I$, and $\mathbf{U}^{0}$ be the Cartan part of $\mathbf{U}$ generated by $K_{\mu}$ for all $\mu\in P^{\vee}$. The quantum group $\mathbf{U}$ has the following triangular decomposition
\begin{displaymath}
\mathbf{U}\cong {\mathbf{U}^-}\otimes{\mathbf{U}^{0}}\otimes{\mathbf{U}^{+}}.
\end{displaymath}


Let $\mathbf{f}$ be the associative algebra defined by Lusztig in \cite{Lusztig_Introduction_to_quantum_groups}. The algebra $\mathbf{f}$ is generated by $\theta_i(i\in I)$ subject to the quantum Serre relations. Let $\mathcal{A}=\mathbb{Z}[v,v^{-1}]$  and $\mathbf{f}_{\mathcal{A}}$ is the integral form of $\mathbf{f}$.
There are two well-defined $\mathbb{Q}(v)$-algebra homomorphisms ${^+}:\mathbf{f}\rightarrow\mathbf{U}$ and ${^-}:\mathbf{f}\rightarrow\mathbf{U}$ satisfying $E_i=\theta_i^+$ and $F_i=\theta_i^-$ for all $i\in I$. The images of ${^+}$ and ${^-}$ are $\mathbf{U}^+$ and $\mathbf{U}^-$ respectively.

\subsection{Lusztig's symmetries}\label{subsection:2.2}

Corresponding to $i\in I$, Lusztig introduced the Lusztig's symmetry $T_i:\mathbf{U}\rightarrow \mathbf{U}$ (\cite{Lusztig_Quantum_deformations_of_certain_simple_modules_over_enveloping_algebras,Lusztig_Quantum_groups_at_roots_of_1,Lusztig_Introduction_to_quantum_groups}).
The formulas of $T_i$ on the generators are:
\begin{eqnarray*}
&&T_i(E_i)=-F_i \tilde{K}_{i},\,\,\,T_i(F_i)=-\tilde{K}_{-i}E_i;\nonumber\\
&&T_i(E_j)=\sum_{r+s=-a_{ij}}(-1)^rv_i^{-r}E_i^{(s)}E_jE_i^{(r)}\,\,\,\textrm{for any $i\neq j\in I$};\label{equation:2.2.1}\\
&&T_i(F_j)=\sum_{r+s=-a_{ij}}(-1)^rv_i^{r}F_i^{(r)}F_jF_i^{(s)}\,\,\,\textrm{for any $i\neq j\in I$};\label{equation:2.2.2}\\
&&T_i(K_{\mu})=K_{\mu-\alpha_{i}(\mu)h_i}\,\,\,\textrm{for any $\mu\in P^{\vee}$},\nonumber
\end{eqnarray*}
where $E_i^{(n)}=E_i^n/[n]_{v_i}!$, $F_i^{(n)}=F_i^n/[n]_{v_i}!$ and $\tilde{K}_{\pm i}=K_{\pm\varepsilon_ih_i}$.

Let ${_i\mathbf{f}}=\{x\in\mathbf{f}\,\,|\,\,T_i(x^+)\in\mathbf{U}^+\}$ and ${^i\mathbf{f}}=\{x\in\mathbf{f}\,\,|\,\,T^{-1}_i(x^+)\in\mathbf{U}^+\}$.
Lusztig symmetry $T_i$ induces a unique map $T_i:{_i\mathbf{f}}\rightarrow{^i\mathbf{f}}$ such that $T_i(x^+)=T_i(x)^+$.

For any $i\neq j\in I$ and $m\in\mathbb{N}$, let
\begin{displaymath}
f(i,j;m)=\sum_{r+s=m}(-1)^rv_i^{-r(-a_{ij}-m+1)}\theta_i^{(r)}\theta_j\theta_i^{(s)}\in\mathbf{f},
\end{displaymath}
and
\begin{displaymath}
f'(i,j;m)=\sum_{r+s=m}(-1)^rv_i^{-r(-a_{ij}-m+1)}\theta_i^{(s)}\theta_j\theta_i^{(r)}\in\mathbf{f},
\end{displaymath}
where
$\theta_i^{(n)}=\theta_i^n/[n]_{v_i}!$.

\begin{proposition}[\cite{Lusztig_Introduction_to_quantum_groups}]\label{proposition:2.1}
For any $i\in I$,\\
(1) ${_i\mathbf{f}}$ (resp. ${^i\mathbf{f}}$) is the subalgebra of $\mathbf{f}$ generated by $f(i,j;m)$ (resp. $f'(i,j;m)$) for all $i\neq j\in I$ and $-a_{ij}\geq m\in\mathbb{N}$;\\
(2) $T_i:{_i\mathbf{f}}\rightarrow{^i\mathbf{f}}$ is an isomorphism of algebras and $$T_i(f(i,j;m))=f'(i,j;-a_{ij}-m)$$ for all $i\neq j\in I$ and $-a_{ij}\geq m\in\mathbb{N}$.
\end{proposition}


Lusztig also showed that $\mathbf{f}$ has the following direct sum decompositions
$$\mathbf{f}={_i\mathbf{f}}\bigoplus\theta_i\mathbf{f}={^i\mathbf{f}}\bigoplus\mathbf{f}\theta_i.$$
Denote by $_i\pi:\mathbf{f}\rightarrow{_i\mathbf{f}}$ and $^i\pi:\mathbf{f}\rightarrow{^i\mathbf{f}}$ the natural projections.

%

\section{Geometric realization of $\mathbf{f}$}\label{section:3}

In this section, we shall review the geometric realization of $\mathbf{f}$ given by Lusztig in \cite{Lusztig_Canonical_bases_arising_from_quantized_enveloping_algebra,Lusztig_Quivers_perverse_sheaves_and_the_quantized_enveloping_algebras,Lusztig_Introduction_to_quantum_groups,Lusztig_Canonical_bases_and_Hall_algebras}.

\subsection{Quivers with automorphisms}\label{subsection:3.1}

Let $Q=(\mathbf{I},H,s,t)$ be a quiver, where $\mathbf{I}$ is the set of vertices, $H$ is the set of arrows, and $s,t:H\rightarrow \mathbf{I}$ are two maps  such that an arrow $\rho\in H$ starts at $s(\rho)$ and terminates at $t(\rho)$.

An admissible automorphism $a$ of $Q$ consists of permutations $a:\mathbf{I}\rightarrow\mathbf{I}$ and $a:H\rightarrow H$ satisfying the following conditions:
\begin{enumerate}
\item[(1)]for any $h\in H$, $s(a(h))=a(s(h))$ and $t(a(h))=a(t(h))$;
\item[(2)]there are no arrows between two vertices in the same $a$-orbit.
\end{enumerate}
From now on, $\tilde{Q}=(Q,a)$ is called a quiver with automorphism.
Assume that $a^\mathbf{n}=\mathrm{id}$ for a given positive integer $\mathbf{n}$.

Let $I=\mathbf{I}^a$. For any $i,j\in I$, let
$$a_{ij}=\left\{\begin{array}{c}
           -|\{\mathbf{i}\rightarrow\mathbf{j}\,\,|\,\,\mathbf{i}\in i, \mathbf{j}\in j\}|
           -|\{\mathbf{j}\rightarrow\mathbf{i}\,\,|\,\,\mathbf{i}\in i, \mathbf{j}\in j\}|,\,\,\,\textrm{if $i\neq j$;}\\
           2|i|,\,\,\,\textrm{if $i=j$.}
         \end{array}
\right.$$
The matrix $A=(a_{ij})_{i,j\in I}$ is a
symmetrizable generalized Cartan matrix.

\begin{proposition}[\cite{Lusztig_Introduction_to_quantum_groups}]\label{proposition:3.1}
For any symmetrizable generalized Cartan matrix $A$, there exists a quiver with automorphism $\tilde{Q}$, such that the generalized Cartan matrix correspongding to $\tilde{Q}$ is $A$.
\end{proposition}

\subsection{Geometric realization of Lusztig's algebra $\hat{\mathbf{f}}$ corresponding to $Q$}\label{subsection:3.2}

Let $p$ be a prime and $q=p^e$. Denote by $\mathbb{F}_q$ the finite field with $q$ elements and $\mathbb{K}=\overline{\mathbb{F}}_q$.

Let $Q=(\mathbf{I},H,s,t)$ be a quiver. Consider the category $\mathcal{C}'$, whose objects are finite dimensional $\mathbf{I}$-graded $\mathbb{K}$-vector spaces $\mathbf{V}=\bigoplus_{\mathbf{i}\in \mathbf{I}}V_\mathbf{i}$, and morphisms are graded linear maps. For any $\nu\in\mathbb{N}\mathbf{I}$,
let $\mathcal{C}'_{\nu}$ be the subcategory of $\mathcal{C}'$ consisting of the objects $\mathbf{V}=\bigoplus_{\mathbf{i}\in \mathbf{I}}V_\mathbf{i}$ such that the dimension vector $\underline{\dim}\mathbf{V}=\sum_{\mathbf{i}\in \mathbf{I}}(\dim_{\mathbb{K}}V_\mathbf{i})\mathbf{i}=\nu$.

For any $\mathbf{V}\in\mathcal{C}'$, define
$$E_\mathbf{V}=\bigoplus_{\rho\in H}\textrm{Hom}_{\mathbb{K}}(V_{s(\rho)},V_{t(\rho)}).$$
The algebraic group $G_{\mathbf{V}}=\prod_{\mathbf{i}\in \mathbf{I}}GL_{\mathbb{K}}(V_\mathbf{i})$ acts on $E_\mathbf{V}$ naturally.


For any $\nu=\nu_{\mathbf{i}}\mathbf{i}\in\mathbb{N}\mathbf{I}$, $\nu$ is called discrete if there is no $h\in H$ such that $\{s(h),t(h)\}\in\{\mathbf{i}\in\mathbf{I}\,\,|\,\,\nu_{\mathbf{i}}\neq 0\}$.
Fix a nonzero element $\nu\in\mathbb{N}\mathbf{I}$.
Let
$$Y_{\nu}=\{\mathbf{y}=(\nu^1,\nu^2,\ldots,\nu^k)\,\,|\,\,\nu^l\in\mathbb{N}\mathbf{I}\textrm{ is discrete and }\sum_{l=1}^{k}\nu^l=\nu\}.$$
Fix $\mathbf{V}\in\mathcal{C}'_{\nu}$.
For any element $\mathbf{y}\in Y_{\nu}$,
a flag of type $\mathbf{y}$ in $\mathbf{V}$ is a sequence
$$\phi=(\mathbf{V}=\mathbf{V}^k\supset\mathbf{V}^{k-1}\supset\dots\supset\mathbf{V}^0=0)$$
where $\mathbf{V}^l\in\mathcal{C}'$ such that $\underline{\dim}\mathbf{V}^l/\mathbf{V}^{l-1}=\nu^l$.
Let $F_{\mathbf{y}}$ be the variety of all flags of type $\mathbf{y}$ in $\mathbf{V}$.
For any $x\in E_{\mathbf{V}}$, a flag $\phi$ is called $x$-stable if $x_{\rho}(V^l_{s(\rho)})\subset{V}^l_{t(\rho)}$ for all $l$ and all $\rho\in H$. Let
$$\tilde{F}_{\mathbf{y}}=\{(x,\phi)\in E_{\mathbf{V}}\times F_{\mathbf{y}}\,\,|\,\,\textrm{$\phi$ is $x$-stable}\}$$
and $\pi_{\mathbf{y}}:\tilde{F}_{\mathbf{y}}\rightarrow E_{\mathbf{V}}$
be the projection to $E_{\mathbf{V}}$.

Let
$\bar{\mathbb{Q}}_{l}$ be the $l$-adic field and
$\mathcal{D}_{G_{\mathbf{V}}}(E_{\mathbf{V}})$ be the bounded $G_{\mathbf{V}}$-equivariant derived category of complexes of $l$-adic sheaves on $E_{\mathbf{V}}$. For each $\mathbf{y}\in{Y}_{\nu}$, $\mathcal{L}_{\mathbf{y}}=\pi_{\mathbf{y}!}\mathbf{1}_{\tilde{F}_{\mathbf{y}}}[d_{\mathbf{y}}](\frac{d_{\mathbf{y}}}{2})\in\mathcal{D}_{G_{\mathbf{V}}}(E_{\mathbf{V}})$ is a semisimple perverse sheaf, where $d_{\mathbf{y}}=\dim\tilde{F}_{\mathbf{y}}$.
Let $\mathcal{P}_{\mathbf{V}}$ be the set of isomorphism classes of simple perverse sheaves $\mathcal{L}$ on $E_{\mathbf{V}}$ such that $\mathcal{L}[r](\frac{r}{2})$ appears as a direct summand of $\mathcal{L}_{\mathbf{y}}$ for some $\mathbf{y}\in {Y}_{\nu}$ and $r\in\mathbb{Z}$. Let $\mathcal{Q}_{\mathbf{V}}$ be the full subcategory of $\mathcal{D}_{G_{\mathbf{V}}}(E_{\mathbf{V}})$ consisting of all complexes which are isomorphic to finite direct sums of complexes in the set
$\{\mathcal{L}[r](\frac{r}{2})\,\,|\,\,\mathcal{L}\in\mathcal{P}_{\mathbf{V}},r\in\mathbb{Z}\}$.

Let $K(\mathcal{Q}_{\mathbf{V}})$ be the Grothendieck group of $\mathcal{Q}_{\mathbf{V}}$.
Define $$v^{\pm}[\mathcal{L}]=[\mathcal{L}[\pm1](\pm\frac{1}{2})].$$
Then, $K(\mathcal{Q}_{\mathbf{V}})$ is a free $\mathcal{A}$-module.
Define $$K(\mathcal{Q})=\bigoplus_{\nu\in\mathbb{N}I}K(\mathcal{Q}_{\mathbf{V}}).$$


For any $\nu,\nu',\nu''\in\mathbb{N}\mathbf{I}$ such that $\nu=\nu'+\nu''$, fix $\mathbf{V}\in\mathcal{C}'_{\nu}$, $\mathbf{V}'\in\mathcal{C}'_{\nu'}$, $\mathbf{V}''\in\mathcal{C}'_{\nu''}$. Consider the following diagram
\begin{equation*}
\xymatrix{E_{\mathbf{V}'}\times E_{\mathbf{V}''}&E'\ar[l]_-{p_1}\ar[r]^-{p_2}&E''\ar[r]^-{p_3}&E_{\mathbf{V}}}
\end{equation*}
where
\begin{enumerate}
  \item[(1)]$E''=\{(x,\mathbf{W})\}$, where $x\in E_{\mathbf{V}}$ and $\mathbf{W}\in\mathcal{C}'_{\nu}$ is an $x$-stable subspace of $\mathbf{V}$;
  \item[(2)]$E'=\{(x,\mathbf{W},R'',R')\}$, where $(x,\mathbf{W})\in E''$, $R'':\mathbf{V}''\simeq\mathbf{W}$ and $R':\mathbf{V}'\simeq\mathbf{V}/\mathbf{W}$;
  \item[(3)]$p_1(x,\mathbf{W},R'',R')=(x',x'')$, where $x'$ and $x''$ are induced through the following commutative diagrams
      $$\xymatrix{\mathbf{V}'_{s(\rho)}\ar[r]^-{x'_{\rho}}\ar[d]^-{R'_{s(\rho)}}&\mathbf{V}'_{t(\rho)}\ar[d]^-{R'_{t(\rho)}}\\
      (\mathbf{V}/\mathbf{W})_{s(\rho)}\ar[r]^-{x_{\rho}}&(\mathbf{V}/\mathbf{W})_{t(\rho)}}$$
      $$\xymatrix{\mathbf{V}''_{s(\rho)}\ar[r]^-{x''_{\rho}}\ar[d]^-{R''_{s(\rho)}}&\mathbf{V}''_{t(\rho)}\ar[d]^-{R''_{t(\rho)}}\\
      \mathbf{W}_{s(\rho)}\ar[r]^-{x_{\rho}}&\mathbf{W}_{t(\rho)}}$$
  \item[(4)]$p_2(x,\mathbf{W},R'',R')=(x,\mathbf{W})$;
  \item[(5)]$p_3(x,\mathbf{W})=x$.
\end{enumerate}

For any two complexes $\mathcal{L}'\in\mathcal{D}_{G_{\mathbf{V}'}}(E_{\mathbf{V}'})$ and $\mathcal{L}''\in\mathcal{D}_{G_{\mathbf{V}''}}(E_{\mathbf{V}''})$, $\mathcal{L}=\mathcal{L}'\ast\mathcal{L}''$ is defined as follows.

Let $\mathcal{L}_1=\mathcal{L}'\otimes\mathcal{L}''$ and $\mathcal{L}_2=p_1^{\ast}\mathcal{L}_1$. Since $p_1$ is smooth with connected fibres and $p_2$ is a $G_{\mathbf{V}'}\times G_{\mathbf{V}''}$-principal bundle, there exists a complex $\mathcal{L}_3$ on $E'$ such that $p_2^{\ast}(\mathcal{L}_3)=\mathcal{L}_2$. The complex $\mathcal{L}$ is defined as $(p_3)_{!}\mathcal{L}_3$.

\begin{lemma}[\cite{Lusztig_Quivers_perverse_sheaves_and_the_quantized_enveloping_algebras,Lusztig_Introduction_to_quantum_groups}]\label{lemma:3.2}
For any $\mathcal{L}'\in\mathcal{Q}_{\mathbf{V}'}$ and $\mathcal{L}''\in\mathcal{Q}_{\mathbf{V}''}$, $\mathcal{L}'\ast\mathcal{L}''\in\mathcal{Q}_{\mathbf{V}}$.
\end{lemma}

Hence, we get a functor
$$\ast:\mathcal{Q}_{\mathbf{V}'}\times\mathcal{Q}_{\mathbf{V}''}\rightarrow\mathcal{Q}_{\mathbf{V}}.$$
This functor induces an associative $\mathcal{A}$-bilinear multiplication
\begin{eqnarray*}
\circledast:K(\mathcal{Q}_{\mathbf{V}'})\times K(\mathcal{Q}_{\mathbf{V}''})&\rightarrow&K(\mathcal{Q}_{\mathbf{V}})\\
([\mathcal{L}']\,,\,[\mathcal{L}''])&\mapsto&[\mathcal{L}']\circledast[\mathcal{L}'']=[\mathcal{L}'\circledast\mathcal{L}'']
\end{eqnarray*}
where $\mathcal{L}'\circledast\mathcal{L}''=(\mathcal{L}'\ast\mathcal{L}'')[m_{\nu'\nu''}](\frac{m_{\nu'\nu''}}{2})$
and $m_{\nu'\nu''}=\sum_{\rho\in H}\nu'_{s(\rho)}\nu''_{t(\rho)}-\sum_{i\in I}\nu'_i\nu''_i$.
Then $K(\mathcal{Q})$ becomes an associative $\mathcal{A}$-algebra and the set $\{[\mathcal{L}]\,\,|\,\,\mathcal{L}\in\mathcal{P}_{\mathbf{V}}\}$ is a basis of $K(\mathcal{Q}_{\mathbf{V}})$.

Let $\hat{\mathbf{f}}$ be the Lusztig's algebra corresponding to $Q$. For any $\mathbf{y}=(a_1\mathbf{i}_1,a_2\mathbf{i}_2,\ldots,a_k\mathbf{i}_k)\in{Y}_{\nu}$, let $\theta_{\mathbf{y}}=\theta_{\mathbf{i}_1}^{(a_1)}\theta_{\mathbf{i}_2}^{(a_2)}\cdots\theta_{\mathbf{i}_k}^{(a_k)}$.

\begin{theorem}[\cite{Lusztig_Quivers_perverse_sheaves_and_the_quantized_enveloping_algebras,Lusztig_Introduction_to_quantum_groups}]\label{theorem:3.3}
There is a unique $\mathcal{A}$-algebra isomorphism
$$\hat{\lambda}_{\mathcal{A}}:K(\mathcal{Q})\rightarrow\hat{\mathbf{f}}_{\mathcal{A}}$$
such that $\hat{\lambda}_{\mathcal{A}}([\mathcal{L}_{\mathbf{y}}])=\theta_{\mathbf{y}}$ for all $\mathbf{y}=(a_1\mathbf{i}_1,a_2\mathbf{i}_2,\ldots,a_k\mathbf{i}_k)\in{Y}_{\nu}$.
\end{theorem}

Let $\hat{\mathbf{B}}_{\nu}=\{[\mathcal{L}]\,\,|\,\,\mathcal{L}\in\mathcal{P}_{\mathbf{V}}\}$ and $\hat{\mathbf{B}}=\bigsqcup_{\nu\in\mathbb{N}\mathbf{I}}\hat{\mathbf{B}}_{\nu}$, which is an $\mathcal{A}$-basis of $K(\mathcal{Q})$ and is called the canonical basis by Lusztig.

\subsection{Geometric realization of $\mathbf{f}$}\label{subsection:3.3}

\subsubsection{}

Let $\tilde{Q}=(Q,a)$ be a quiver with automorphism, where $Q=(\mathbf{I},H,s,t)$.
Let $\tilde{\mathcal{C}}$ be the category of $\mathbf{V}=\bigoplus_{\mathbf{i}\in \mathbf{I}}V_\mathbf{i}\in\mathcal{C}'$
with a linear map $a:\mathbf{V}\rightarrow\mathbf{V}$ satisfying the following conditions:
\begin{enumerate}
  \item[(1)]for any $\mathbf{i}\in\mathbf{I}$, $a(V_\mathbf{i})=V_{a(\mathbf{i})}$;
  \item[(2)]for any $\mathbf{i}\in\mathbf{I}$ and $k\in\mathbb{N}$ such that $a^k(\mathbf{i})=\mathbf{i}$, $a^k|_{V_\mathbf{i}}=\mathrm{id}_{V_\mathbf{i}}$.
\end{enumerate}
The morphisms in $\tilde{\mathcal{C}}$ are the graded linear maps $f=(f_\mathbf{i})_{\mathbf{i}\in\mathbf{I}}$ such that the following diagram commutes
$$
\xymatrix{
{V}_{\mathbf{i}}\ar[r]^-{a}\ar[d]^-{f_\mathbf{i}}&{V}_{a(\mathbf{i})}\ar[d]^-{f_{a(\mathbf{i})}}\\
{V}_{\mathbf{i}}\ar[r]^-{a}&{V}_{a(\mathbf{i})}
}
$$
Let $\mathbb{N}\mathbf{I}^a=\{\nu\in\mathbb{N}\mathbf{I}\,\,|\,\,\nu_\mathbf{i}=\nu_{a(\mathbf{i})}\}$. There is a bijection between $\mathbb{N}I$ and $\mathbb{N}\mathbf{I}^a$ sending $i$ to $\gamma_i=\sum_{\mathbf{i}\in i}\mathbf{i}$. From now on, $\mathbb{N}\mathbf{I}^a$ is identified with $\mathbb{N}I$.
For any $\nu\in\mathbb{N}\mathbf{I}^a$,
let $\tilde{\mathcal{C}}_{\nu}$ be the subcategory of $\tilde{\mathcal{C}}$ consisting of the objects $\mathbf{V}=\bigoplus_{\mathbf{i}\in \mathbf{I}}V_\mathbf{i}$ such that the dimension vector $\underline{\dim}\mathbf{V}=\nu$.

For any $(\mathbf{V},a)\in\tilde{\mathcal{C}}$, $E_{\mathbf{V}}$ and $G_{\mathbf{V}}$ are defined in Section \ref{subsection:3.2}.
Let $a:G_{\mathbf{V}}\rightarrow G_{\mathbf{V}}$ be the automorphism defined by $$a(g)(v)=a(g(a^{-1}(v)))$$ for any $g\in G_{\mathbf{V}}$ and $v\in\mathbf{V}$.
Denote by
$a:E_{\mathbf{V}}\rightarrow E_{\mathbf{V}}$ the automorphism such that the following diagram commutes for any $h\in H$
$$
\xymatrix{
\mathbf{V}_{s(h)}\ar[r]^-{x_h}\ar[d]^-{a}&\mathbf{V}_{t(h)}\ar[d]^-{a}\\
\mathbf{V}_{a(s(h))}\ar[r]^-{a(x)_{a(h)}}&\mathbf{V}_{a(t(h))}
}
$$

Since $a(gx)=a(g)a(x)$, we have a functor $a^{\ast}:\mathcal{D}_{G_{\mathbf{V}}}(E_{\mathbf{V}})\rightarrow\mathcal{D}_{G_{\mathbf{V}}}(E_{\mathbf{V}})$.

\begin{lemma}[\cite{Lusztig_Introduction_to_quantum_groups}]\label{lemma:3.4}
It holds that $a^{\ast}(\mathcal{Q}_{\mathbf{V}})=\mathcal{Q}_{\mathbf{V}}$ and $a^{\ast}(\mathcal{P}_{\mathbf{V}})=\mathcal{P}_{\mathbf{V}}$.
\end{lemma}

Lusztig introduced the following categories $\tilde{\mathcal{Q}}_{\mathbf{V}}$ and $\tilde{\mathcal{P}}_{\mathbf{V}}$ in \cite{Lusztig_Introduction_to_quantum_groups}.
The objects in $\tilde{\mathcal{Q}}_{\mathbf{V}}$ are pairs $(\mathcal{L},\phi)$,
where $\mathcal{L}\in\mathcal{Q}_{\mathbf{V}}$ and $\phi:a^{\ast}\mathcal{L}\rightarrow\mathcal{L}$ is an isomorphism such that
$$\mathcal{L}=a^{\ast\mathbf{n}}\mathcal{L}\rightarrow
a^{\ast(\mathbf{n}-1)}\mathcal{L}\rightarrow
\ldots
\rightarrow
a^{\ast}\mathcal{L}\rightarrow\mathcal{L}$$
is the identity map of $\mathcal{L}$. A morphism in $\textrm{Hom}_{\tilde{\mathcal{Q}}_{\mathbf{V}}}((\mathcal{L},\phi),(\mathcal{L}',\phi'))$ is
a morphism $f\in\textrm{Hom}_{\mathcal{Q}_{\mathbf{V}}}(\mathcal{L},\mathcal{L}')$
such that the following diagram commutes
$$
\xymatrix{
a^{\ast}\mathcal{L}\ar[r]^-{\phi}\ar[d]^-{a^{\ast}f}&\mathcal{L}\ar[d]^-{f}\\
a^{\ast}\mathcal{L}'\ar[r]^-{\phi'}&\mathcal{L}'
}
$$

Let $\mathcal{O}$ be the subring of $\bar{\mathbb{Q}}_{l}$ consisting of all $\mathbb{Z}$-linear combinations of $\mathbf{n}$-th roots of $1$.
The Grothendieck groups $K(\tilde{\mathcal{Q}}_{\mathbf{V}})$ and $K(\tilde{\mathcal{P}}_{\mathbf{V}})$ of $\tilde{\mathcal{Q}}_{\mathbf{V}}$ and $\tilde{\mathcal{P}}_{\mathbf{V}}$ were introduced by Lusztig in \cite{Lusztig_Introduction_to_quantum_groups}, which are $\mathcal{O}$-modules.
Let $\mathcal{O}'=\mathcal{O}[v,v^{-1}]$.
Since $a^{\ast}$ commutes with the shift functor and the Tate twist, we can define
$$v^{\pm}[\mathcal{L},\phi]=[\mathcal{L}[\pm1](\pm\frac{1}{2}),\phi[\pm1](\pm\frac{1}{2})].$$
Then $K(\tilde{\mathcal{Q}}_{\mathbf{V}})$ has a natural $\mathcal{O}'$-module structure.
Note that $K(\tilde{\mathcal{Q}}_{\mathbf{V}})=\mathcal{O}'\otimes_{\mathcal{O}}K(\tilde{\mathcal{P}}_{\mathbf{V}})$.
Define $$K(\tilde{\mathcal{Q}})=\bigoplus_{\nu\in\mathbb{N}I}K(\tilde{\mathcal{Q}}_{\mathbf{V}}).$$

For $\nu,\nu',\nu''\in\mathbb{N}I=\mathbb{N}\mathbf{I}^a$ such that $\nu=\nu'+\nu''$, fix $\mathbf{V}\in\tilde{\mathcal{C}}_{\nu}$, $\mathbf{V}'\in\tilde{\mathcal{C}}_{\nu'}$, $\mathbf{V}''\in\tilde{\mathcal{C}}_{\nu''}$. Lusztig proved the following lemma.

\begin{lemma}[\cite{Lusztig_Introduction_to_quantum_groups}]\label{lemma:3.5}
The induction functor
$$\ast:\mathcal{Q}_{\mathbf{V}'}\times\mathcal{Q}_{\mathbf{V}''}\rightarrow\mathcal{Q}_{\mathbf{V}}$$ satisfies that the following diagram commutes
$$
\xymatrix{
\mathcal{Q}_{\mathbf{V}'}\times\mathcal{Q}_{\mathbf{V}''}\ar[r]^-{\ast}\ar[d]^-{a^{\ast}\times a^{\ast}}&\mathcal{Q}_{\mathbf{V}}\ar[d]^-{a^{\ast}}\\
\mathcal{Q}_{\mathbf{V}'}\times\mathcal{Q}_{\mathbf{V}''}\ar[r]^-{\ast}&\mathcal{Q}_{\mathbf{V}}
}
$$
\end{lemma}

For any $(\mathcal{L}',\phi')\in\tilde{\mathcal{Q}}_{\mathbf{V}'}$ and $(\mathcal{L}'',\phi'')\in\tilde{\mathcal{Q}}_{\mathbf{V}''}$, Lemma \ref{lemma:3.5} implies that $$a^{\ast}(\mathcal{L}'\ast\mathcal{L}'')=a^{\ast}\mathcal{L}'\ast a^{\ast}\mathcal{L}''.$$
Hence there is a functor
\begin{eqnarray*}
\ast:\tilde{\mathcal{Q}}_{\mathbf{V}'}\times\tilde{\mathcal{Q}}_{\mathbf{V}''}&\rightarrow&\tilde{\mathcal{Q}}_{\mathbf{V}}\\
((\mathcal{L}',\phi'),(\mathcal{L}'',\phi''))&\mapsto&(\mathcal{L}'\ast\mathcal{L}'',\phi)
\end{eqnarray*}
where $$\phi=\phi'\ast\phi'':\mathcal{L}'\ast\mathcal{L}''\rightarrow a^{\ast}\mathcal{L}'\ast a^{\ast}\mathcal{L}''=a^{\ast}(\mathcal{L}'\ast\mathcal{L}'').$$
This functor induces an associative $\mathcal{O}'$-bilinear multiplication
\begin{eqnarray*}
\circledast:K(\tilde{\mathcal{Q}}_{\mathbf{V}'})\times K(\tilde{\mathcal{Q}}_{\mathbf{V}''})&\rightarrow&K(\tilde{\mathcal{Q}}_{\mathbf{V}})\\
([\mathcal{L}',\phi']\,,\,[\mathcal{L}'',\phi''])&\mapsto&[\mathcal{L},\phi]
\end{eqnarray*}
where $(\mathcal{L},\phi)=(\mathcal{L}',\phi')\circledast(\mathcal{L}'',\phi'')=((\mathcal{L}',\phi')\ast(\mathcal{L}'',\phi''))[m_{\nu'\nu''}](\frac{m_{\nu'\nu''}}{2})$.
Then $K(\tilde{\mathcal{Q}})$ becomes an associative $\mathcal{O}'$-algebra.

\subsubsection{}

Fix a nonzero element $\nu\in\mathbb{N}\mathbf{I}^a$.
Let
$$Y^a_{\nu}=\{\mathbf{y}=(\nu^1,\nu^2,\ldots,\nu^k)\in Y_{\nu}\,\,|\,\,\nu^l\in\mathbb{N}\mathbf{I}^a\}.$$
Fix $\mathbf{V}\in\tilde{\mathcal{C}}_{\nu}$.
For any element $\mathbf{y}\in Y^a_{\nu}$,
the automorphism $a:F_{\mathbf{y}}\rightarrow F_{\mathbf{y}}$ is defined as
$$a(\phi)=(\mathbf{V}=a(\mathbf{V}^k)\supset a(\mathbf{V}^{k-1})\supset\dots\supset a(\mathbf{V}^0)=0)$$
for any $$\phi=(\mathbf{V}=\mathbf{V}^k\supset\mathbf{V}^{k-1}\supset\dots\supset\mathbf{V}^0=0)\in F_{\mathbf{y}}.$$
There also exists an automorphism $a:\tilde{F}_{\mathbf{y}}\rightarrow \tilde{F}_{\mathbf{y}}$, defined as $a((x,\phi))=(a(x),a(\phi))$ for any $(x,\phi)\in\tilde{F}_{\mathbf{y}}$.

Since $a^{\ast}\mathbf{1}_{\tilde{F}_{\mathbf{y}}}\cong\mathbf{1}_{\tilde{F}_{\mathbf{y}}}$, there exists an isomprphism
$$\phi_0:a^{\ast}\mathcal{L}_{\mathbf{y}}=a^{\ast}\pi_{\mathbf{y}!}\mathbf{1}_{\tilde{F}_{\mathbf{y}}}[d_{\mathbf{y}}](\frac{d_{\mathbf{y}}}{2})=
\pi_{\mathbf{y}!}a^{\ast}\mathbf{1}_{\tilde{F}_{\mathbf{y}}}[d_{\mathbf{y}}](\frac{d_{\mathbf{y}}}{2})
\cong\pi_{\mathbf{y}!}\mathbf{1}_{\tilde{F}_{\mathbf{y}}}[d_{\mathbf{y}}](\frac{d_{\mathbf{y}}}{2})=\mathcal{L}_{\mathbf{y}}.$$
Hence $(\mathcal{L}_{\mathbf{y}},\phi_0)$ is an object in $\tilde{\mathcal{Q}}_{\mathbf{V}}$.

Let $\mathbf{k}_{\nu}$ be the $\mathcal{A}$-submodule of $K(\tilde{\mathcal{Q}}_{\mathbf{V}})$ spanned by $(\mathcal{L}_{\mathbf{y}},\phi_0)$ for all $\mathbf{y}\in Y^a_{\nu}$. Let $\mathbf{k}=\bigoplus_{\nu\in\mathbb{N}I}\mathbf{k}_{\nu}$.
Lusztig proved that $\mathbf{k}$ is also a subalgebra of $K(\tilde{\mathcal{Q}})$.

Let $i\in I$ and $\gamma_i=\sum_{\mathbf{i}\in i}\mathbf{i}$. Define $\mathbf{1}_{i}=[\mathbf{1},\mathrm{id}]\in K(\tilde{\mathcal{Q}}_{\mathbf{V}})$, where $\mathbf{V}\in\tilde{C}_{\gamma_i}$.
Let $\mathbf{f}$ be the Lusztig's algebra corresponding to $(Q,a)$.

\begin{theorem}[\cite{Lusztig_Introduction_to_quantum_groups}]\label{theorem:3.6}
There is a unique $\mathcal{A}$-algebra isomorphism
$$\lambda_{\mathcal{A}}:\mathbf{k}\rightarrow\mathbf{f}_{\mathcal{A}}$$
such that $\lambda_{\mathcal{A}}(\mathbf{1}_{i})=\theta_{i}$ for all $i
\in I$.
\end{theorem}

On the canonical basis of $\mathbf{f}$, Lusztig gave the following theorem.

\begin{theorem}[\cite{Lusztig_Introduction_to_quantum_groups}]\label{theorem:3.7}
(1) For any $B\in\mathcal{P}_{\mathbf{V}}$ such that $a^{\ast}B\cong B$, there exists an isomorphism $\phi:a^{\ast}B\cong B$ such that $(B,\phi)\in\tilde{\mathcal{P}}_{\mathbf{V}}$ and $(D(B),D(\phi)^{-1})$ is isomorphic to $(B,\phi)$ as objects of $\tilde{P}_{\mathbf{V}}$, where $D$ is the Verdier duality. Moreover, $\phi$ is unique, if $\mathbf{n}$ is odd, and unique up to multiplication by $\pm1$, if $\mathbf{n}$ is even.\\
(2) $\mathbf{k}_{\nu}$ is generated by $[B,\phi]$ in (1) as an $\mathcal{A}$-submodule of $K(\tilde{\mathcal{Q}}_{\mathbf{V}})$.
\end{theorem}

If $\mathbf{n}$ is odd, let $\mathcal{B}_{\nu}$ be the subset of $\mathbf{k}_{\nu}$ consisting of all elements $[B,\phi]$ in Theorem \ref{theorem:3.7}.
If $\mathbf{n}$ is even, let $\mathcal{B}_{\nu}$ be the subset of $\mathbf{k}_{\nu}$ consisting of all elements $\pm[B,\phi]$ in Theorem \ref{theorem:3.7}.
The set $\mathcal{B}_{\nu}$ is called a signed basis of $\mathbf{k}_{\nu}$ by Lusztig. Lusztig gave a non-geometric way to choose a subset $\mathbf{B}_{\nu}$ of $\mathcal{B}_{\nu}$ such that $\mathcal{B}_{\nu}=\mathbf{B}_{\nu}\cup-\mathbf{B}_{\nu}$ and $\mathbf{B}_{\nu}$ is an $\mathcal{A}$-basis of $\mathbf{k}_{\nu}$ (\cite{Lusztig_Introduction_to_quantum_groups}). The set $\mathbf{B}=\sqcup_{\nu\in\mathbb{N}I}\mathbf{B}_{\nu}$ is called the canonical basis of $\mathbf{k}$.

At last, let us recall the relation between $\hat{\mathbf{f}}$ and $\mathbf{f}$.
Define $\tilde{\delta}:\mathbf{k}\rightarrow K(\mathcal{Q})$ by $\tilde{\delta}([B,\phi])=B$ for any $[B,\phi]\in\mathbf{B}$. It is clear that this is an injection and induces an embedding $\delta:\mathbf{f}\rightarrow\hat{\mathbf{f}}$. Note that $\delta(\mathbf{B})=\hat{\mathbf{B}}^a$.

\section{Geometric realizations of subalgebras ${_i\mathbf{f}}$ and ${^i\mathbf{f}}$}\label{section:4}

\subsection{The algebra ${_{\mathfrak{i}}\hat{\mathbf{f}}}$}\label{subsection:4.1}

Let $Q=(\mathbf{I},H,s,t)$ be a quiver and $\hat{\mathbf{f}}$ be the corresponding Lusztig's algebra.
Let ${_{\mathbf{i}}\hat{\mathbf{f}}}$ be the subalgebra of $\hat{\mathbf{f}}$ generated by $f(\mathbf{i},\mathbf{j};m)$ for all $\mathbf{i}\neq\mathbf{j}\in\mathbf{I}$ and integer $m$.
Let $\mathfrak{i}$ be a subset of $\mathbf{I}$ and
define $${_{\mathfrak{i}}\hat{\mathbf{f}}}=\bigcap_{\mathbf{i}\in\mathfrak{i}} {_{\mathbf{i}}\hat{\mathbf{f}}},$$
which is also a subalgebra of $\hat{\mathbf{f}}$.

For any $\mathbf{i}\in\mathbf{I}$, there exists a unique linear map ${_{\mathbf{i}}r}:\hat{\mathbf{f}}\rightarrow\hat{\mathbf{f}}$ such that ${_{\mathbf{i}}r}(\mathbf{1})=0$, ${_{\mathbf{i}}r}(\theta_\mathbf{j})=\delta_{\mathbf{i}\mathbf{j}}$ for all $\mathbf{j}\in\mathbf{I}$
and ${_{\mathbf{i}}r}(xy)={_{\mathbf{i}}r}(x)y+v^{(\nu,\alpha_\mathbf{i})}x{_{\mathbf{i}}r}(y)$ for all homogeneous $x\in\hat{\mathbf{f}}_{\nu}$ and $y$.
Denote by $(-,-)$ the non-degenerate symmetric bilinear form on $\hat{\mathbf{f}}$ introduced by Lusztig.

\begin{proposition}[\cite{Lusztig_Introduction_to_quantum_groups}]\label{proposition:4.1}
It holds that ${_{\mathbf{i}}\hat{\mathbf{f}}}=\{x\in\hat{\mathbf{f}}\,\,|\,\,{_{\mathbf{i}}r}(x)=0\}$.
\end{proposition}

As a corollary of Proposition \ref{proposition:4.1}, we have

\begin{corollary}\label{corollary:4.2}
It holds that ${_{\mathfrak{i}}\hat{\mathbf{f}}}=\{x\in\hat{\mathbf{f}}\,\,|\,\,{_{\mathbf{i}}r}(x)=0\,\,\,\textrm{for any $\mathbf{i}\in\mathfrak{i}$}\}$.
\end{corollary}
\qed

%

\begin{proposition}\label{proposition:4.3}
The algebra $\hat{\mathbf{f}}$ has the following decomposition
$$\hat{\mathbf{f}}={_{\mathfrak{i}}\hat{\mathbf{f}}}\oplus\sum_{\mathbf{i}\in\mathfrak{i}}\theta_{\mathbf{i}}\hat{\mathbf{f}}.$$
\end{proposition}

\begin{proof}

In the algebra $\hat{\mathbf{f}}$,
$(\theta_{\mathbf{i}}y,x)=(\theta_{\mathbf{i}},\theta_{\mathbf{i}})(y,{_{\mathbf{i}}r}(x))$.
Hence the decomposition $\hat{\mathbf{f}}_{\nu}={_{\mathbf{i}}\hat{\mathbf{f}}}_{\nu}\oplus\theta_{\mathbf{i}}\hat{\mathbf{f}}_{\nu-\mathbf{i}}$ is an orthogonal decomposition and $\theta_{\mathbf{i}}\hat{\mathbf{f}}_{\nu-\mathbf{i}}={_{\mathbf{i}}\hat{\mathbf{f}}}_{\nu}^\bot$.


For the proof of this proposition, it is sufficient to show that
${_{\mathfrak{i}}\hat{\mathbf{f}}}_{\nu}^\bot=\sum_{\mathbf{i}\in\mathfrak{i}}{_{\mathbf{i}}\hat{\mathbf{f}}}_{\nu}^\bot$.
That is $(\bigcap_{\mathbf{i}\in\mathfrak{i}}{_{\mathbf{i}}\hat{\mathbf{f}}}_{\nu})^\bot=\sum_{\mathbf{i}\in\mathfrak{i}}{_{\mathbf{i}}\hat{\mathbf{f}}}_{\nu}^\bot$.
It is clear that $(\bigcap_{\mathbf{i}\in\mathfrak{i}}{_{\mathbf{i}}\hat{\mathbf{f}}}_{\nu})^\bot\supset\sum_{\mathbf{i}\in\mathfrak{i}}{_{\mathbf{i}}\hat{\mathbf{f}}}_{\nu}^\bot$.
On the other hand,
$x\in(\sum_{\mathbf{i}\in\mathfrak{i}}{_{\mathbf{i}}\hat{\mathbf{f}}}_{\nu}^\bot)^{\bot}$ implies 
$x\in\bigcap_{\mathbf{i}\in\mathfrak{i}}{_{\mathbf{i}}\hat{\mathbf{f}}}_{\nu}$.
Hence
$\bigcap_{\mathbf{i}\in\mathfrak{i}}{_{\mathbf{i}}\hat{\mathbf{f}}}_{\nu}\supset(\sum_{\mathbf{i}\in\mathfrak{i}}{_{\mathbf{i}}\hat{\mathbf{f}}}_{\nu}^\bot)^{\bot}$.
That is
$(\bigcap_{\mathbf{i}\in\mathfrak{i}}{_{\mathbf{i}}\hat{\mathbf{f}}}_{\nu})^\bot\subset\sum_{\mathbf{i}\in\mathfrak{i}}{_{\mathbf{i}}\hat{\mathbf{f}}}_{\nu}^\bot$.


\end{proof}

Denoted by $_{\mathfrak{i}}\pi:\hat{\mathbf{f}}\rightarrow{_{\mathfrak{i}}\hat{\mathbf{f}}}$ the canonical projection.

\subsection{Geometric realization of ${_{\mathfrak{i}}\hat{\mathbf{f}}}$}\label{subsection:4.2}

\subsubsection{}

Let $\mathfrak{i}$ be a subset of $\mathbf{I}$ satisfying the following conditions:
(1) for any $\mathbf{i}\in\mathfrak{i}$, $\mathbf{i}$ is a sink; (2) for any $\mathbf{i},\mathbf{j}\in\mathfrak{i}$, there are no arrows between them.

For any $\mathbf{V}\in\mathcal{C}_{\nu}'$, consider a subvariety ${_{\mathfrak{i}}E}_\mathbf{V}$ of $E_\mathbf{V}$
$$
{_{\mathfrak{i}}E}_{\mathbf{V}}=\{x\in E_{\mathbf{V}}\,\,|\,\,\bigoplus_{h\in H,t(h)=\mathbf{i}}x_{h}:\bigoplus_{h\in H,t(h)=\mathbf{i}}V_{s(h)}\rightarrow V_\mathbf{i} \,\,\,\textrm{is surjective for any $\mathbf{i}\in\mathfrak{i}$}\}.
$$
Denote by ${_{\mathfrak{i}}j}_{\mathbf{V}}:{_{\mathfrak{i}}E}_{\mathbf{V}}\rightarrow E_{\mathbf{V}}$ the canonical embedding.

For any $\mathbf{y}\in{Y}_{\nu}$, let
\begin{displaymath}
{_{\mathfrak{i}}\tilde{F}_{\mathbf{y}}}=\{(x,\phi)\in {_{\mathfrak{i}}E}_{\mathbf{V}}\times F_{\mathbf{y}}\,\,|\,\,\textrm{$\phi$ is $x$-stable}\}
\end{displaymath}
and
${_{\mathfrak{i}}\pi}_{\mathbf{y}}:{_{\mathfrak{i}}\tilde{F}}_{\mathbf{y}}\rightarrow {_{\mathfrak{i}}E}_{\mathbf{V}}$
be the projection to ${_{\mathfrak{i}}E}_{\mathbf{V}}$.

For any $\mathbf{y}\in{Y}_{\nu}$,
$_{\mathfrak{i}}\mathcal{L}_{\mathbf{y}}={_{\mathfrak{i}}\pi}_{\mathbf{y}!}\mathbf{1}_{{_{\mathfrak{i}}\tilde{F}}_{\mathbf{y}}}[d_{\mathbf{y}}](\frac{d_{\mathbf{y}}}{2})\in\mathcal{D}_{G_{\mathbf{V}}}(_{\mathfrak{i}}E_{\mathbf{V}})$ is a semisimple perverse sheaf.
Let ${_{\mathfrak{i}}\mathcal{P}}_{\mathbf{V}}$ be the set of isomorphism classes of simple perverse sheaves $\mathcal{L}$ on ${_{\mathfrak{i}}E}_{\mathbf{V}}$ such that $\mathcal{L}[r](\frac{r}{2})$ appears as a direct summand of $_{\mathfrak{i}}\mathcal{L}_{\mathbf{y}}$ for some $\mathbf{y}\in{Y}_{\nu}$ and $r\in\mathbb{Z}$. Let ${_{\mathfrak{i}}\mathcal{Q}}_{\mathbf{V}}$ be the full subcategory of $\mathcal{D}_{G_{\mathbf{V}}}(_{\mathfrak{i}}E_{\mathbf{V}})$ consisting of all complexes which are isomorphic to finite direct sums of complexes in the set $\{\mathcal{L}[r](\frac{r}{2})\,\,|\,\,\mathcal{L}\in{_{\mathfrak{i}}\mathcal{P}}_{\mathbf{V}},r\in\mathbb{Z}\}$.

Let $K({_{\mathfrak{i}}\mathcal{Q}}_{\mathbf{V}})$ be the Grothendieck group of ${_{\mathfrak{i}}\mathcal{Q}}_{\mathbf{V}}$.
Define $$v^{\pm}[\mathcal{L}]=[\mathcal{L}[\pm1](\pm\frac{1}{2})].$$
Then, $K({_{\mathfrak{i}}\mathcal{Q}}_{\mathbf{V}})$ is a free $\mathcal{A}$-module.
Define $$K({_{\mathfrak{i}}\mathcal{Q}})=\bigoplus_{\nu\in\mathbb{N}I}K({_{\mathfrak{i}}\mathcal{Q}}_{\mathbf{V}}).$$

The canonical embedding ${_{\mathfrak{i}}j}_{\mathbf{V}}:{_{\mathfrak{i}}E}_{\mathbf{V}}\rightarrow E_{\mathbf{V}}$ induces a functor
$${_{\mathfrak{i}}j}^{\ast}_{\mathbf{V}}:\mathcal{D}_{G_{\mathbf{V}}}(E_{\mathbf{V}})\rightarrow\mathcal{D}_{G_{\mathbf{V}}}(_{\mathfrak{i}}E_{\mathbf{V}}).$$

\begin{lemma}\label{lemma:4.4}
It holds that ${_{\mathfrak{i}}j}^{\ast}_{\mathbf{V}}(\mathcal{Q}_{\mathbf{V}})={_{\mathfrak{i}}\mathcal{Q}}_{\mathbf{V}}$.
\end{lemma}

\begin{proof}
For any $\mathbf{y}\in{Y}_{\nu}$, we have the following fiber product
$$\xymatrix{_{\mathfrak{i}}\tilde{F}_{\mathbf{y}}\ar[r]^-{{_{\mathfrak{i}}\tilde{j}}_\mathbf{V}}\ar[d]^-{_{\mathfrak{i}}\pi_{\mathbf{y}}}&\tilde{F}_{\mathbf{y}}
\ar[d]^-{\pi_{\mathbf{y}}}\\_{\mathfrak{i}}E_{\mathbf{V}}\ar[r]^-{{_{\mathfrak{i}}j}_\mathbf{V}}&E_{\mathbf{V}}}$$
Hence we have
\begin{eqnarray*}
{_{\mathfrak{i}}j}^{\ast}_\mathbf{V}\mathcal{L}_{\mathbf{y}}&=&{_{\mathfrak{i}}j}^{\ast}_\mathbf{V}\pi_{y!}\mathbf{1}_{{\tilde{F}}_{\mathbf{y}}}[d_{\mathbf{y}}](\frac{d_{\mathbf{y}}}{2})\\
&=&{_{\mathfrak{i}}\pi}_{y!}{{_{\mathfrak{i}}\tilde{j}}}^{\ast}_\mathbf{V}\mathbf{1}_{{\tilde{F}}_{\mathbf{y}}}[d_{\mathbf{y}}](\frac{d_{\mathbf{y}}}{2})\\
&=&{_{\mathfrak{i}}\pi}_{y!}\mathbf{1}_{{_{\mathfrak{i}}\tilde{F}}_{\mathbf{y}}}[d_{\mathbf{y}}](\frac{d_{\mathbf{y}}}{2})={_{\mathfrak{i}}\mathcal{L}}_{\mathbf{y}}.
\end{eqnarray*}
That is ${_{\mathfrak{i}}j}^{\ast}_{\mathbf{V}}(\mathcal{Q}_{\mathbf{V}})={_{\mathfrak{i}}\mathcal{Q}}_{\mathbf{V}}$.

\end{proof}

The restriction of ${_{\mathfrak{i}}j}^{\ast}_{\mathbf{V}}:\mathcal{D}_{G_{\mathbf{V}}}(E_{\mathbf{V}})\rightarrow\mathcal{D}_{G_{\mathbf{V}}}(_{\mathfrak{i}}E_{\mathbf{V}})$
on $\mathcal{Q}_{\mathbf{V}}$ is also denoted by ${_{\mathfrak{i}}j}^{\ast}_{\mathbf{V}}:\mathcal{Q}_{\mathbf{V}}\rightarrow{_{\mathfrak{i}}\mathcal{Q}}_{\mathbf{V}}$.
Considering all dimension vectors, we have ${_{\mathfrak{i}}j}^{\ast}:K(\mathcal{Q})\rightarrow K({_{\mathfrak{i}}\mathcal{Q}})$.

\begin{proposition}\label{proposition:4.5}
There exists an isomorphism of vector spaces ${_\mathfrak{i}\hat{\lambda}}_{\mathcal{A}}:K(_\mathfrak{i}\mathcal{Q})\rightarrow{_\mathfrak{i}\hat{\mathbf{f}}}_{\mathcal{A}}$
such that the following diagram commutes
$$
\xymatrix{
K(\mathcal{Q})\ar[r]^-{{_{\mathfrak{i}}j}^{\ast}}\ar[d]^-{\hat{\lambda}_{\mathcal{A}}}&K(_\mathfrak{i}\mathcal{Q})\ar[d]^-{{_\mathfrak{i}\hat{\lambda}}_{\mathcal{A}}}\\
\hat{\mathbf{f}}_{\mathcal{A}}\ar[r]^-{{_\mathfrak{i}\pi}_{\mathcal{A}}}&{_\mathfrak{i}\hat{\mathbf{f}}}_{\mathcal{A}}
}
$$
\end{proposition}

\begin{proof}
Consider the following set of surjections:
$$\{{_\mathbf{i}\pi}_{\mathcal{A}}:\hat{\mathbf{f}}_{\mathcal{A}}\rightarrow{_\mathbf{i}\hat{\mathbf{f}}}_{\mathcal{A}}\,\,|\,\,\mathbf{i}\in\mathfrak{i}\}.$$
It is clear that the push out of this set is ${_\mathfrak{i}\pi}_{\mathcal{A}}:\hat{\mathbf{f}}_{\mathcal{A}}\rightarrow{_\mathfrak{i}\hat{\mathbf{f}}}_{\mathcal{A}}$.

For any $\mathbf{i}\in\mathfrak{i}$, there exists a canonical open embedding
$${_\mathbf{i}j}_{\mathbf{V}}:{_\mathbf{i}E}_{\mathbf{V}}\rightarrow{E}_{\mathbf{V}}.$$
Since ${_\mathfrak{i}E}_{\mathbf{V}}=\bigcap_{\mathbf{i}\in\mathfrak{i}}{_\mathbf{i}E}_{\mathbf{V}}$, we have the following commutative diagram
$$
\xymatrix{
{_\mathfrak{i}E}_{\mathbf{V}}\ar[r]^-{}
\ar[d]^-{}
&{_\mathbf{i}E}_{\mathbf{V}}\ar[d]^-{{_\mathbf{i}j}_{\mathbf{V}}}\\
{_\mathbf{j}E}_{\mathbf{V}}\ar[r]^-{{_\mathbf{j}j}_{\mathbf{V}}}&{E}_{\mathbf{V}}
}
$$
for any $\mathbf{i},\mathbf{j}\in\mathfrak{i}$. Hence we have the following commutative diagram
$$\xymatrix{K(\mathcal{Q}_{\mathbf{V}})\ar[d]^-{{_\mathbf{j}j}^{\ast}_{\mathbf{V}}}\ar[r]^-{{_\mathbf{i}j}^{\ast}_{\mathbf{V}}}
                & K(_\mathbf{i}\mathcal{Q}_{\mathbf{V}}) \ar[d]^-{}   \\
K(_\mathbf{j}\mathcal{Q}_{\mathbf{V}}) \ar[r]^-{}
                & K(_\mathfrak{i}\mathcal{Q}_{\mathbf{V}})}$$
For any simple perverse sheaf $\mathcal{L}$ such that ${_\mathfrak{i}j}^{\ast}_{\mathbf{V}}\mathcal{L}=0$, we have $$\textrm{supp}(\mathcal{L})\subset{E}_{\mathbf{V}}-{_\mathfrak{i}E}_{\mathbf{V}}
={E}_{\mathbf{V}}-\bigcap_{\mathbf{i}\in\mathfrak{i}}{_\mathbf{i}E}_{\mathbf{V}}
=\bigcup_{\mathbf{i}\in\mathfrak{i}}({E}_{\mathbf{V}}-{_\mathbf{i}E}_{\mathbf{V}}).$$
Hence $\textrm{supp}(\mathcal{L})\subset{E}_{\mathbf{V}}-{_\mathbf{i}E}_{\mathbf{V}}$ for some $\mathbf{i}\in\mathfrak{i}$ (Section 9.3.4 in \cite{Lusztig_Introduction_to_quantum_groups}). So ${_\mathbf{i}j}^{\ast}_{\mathbf{V}}\mathcal{L}=0$.
By the definition of push out, the push out of the following set
$$\{{_{\mathbf{i}}j}^{\ast}:K(\mathcal{Q})\rightarrow K(_\mathbf{i}\mathcal{Q})\,\,|\,\,\mathbf{i}\in\mathfrak{i}\}$$
is ${_{\mathfrak{i}}j}^{\ast}:K(\mathcal{Q})\rightarrow K(_\mathfrak{i}\mathcal{Q})$.

In \cite{Xiao_Zhao_Geometric_realizations_of_Lusztig's_symmetries}, it was proved that the following diagram commutes
$$
\xymatrix{
K(\mathcal{Q})\ar[r]^-{{_\mathbf{i}j}^{\ast}}\ar[d]^-{\hat{\lambda}_{\mathcal{A}}}&K(_\mathbf{i}\mathcal{Q})\ar[d]^-{{_{\mathbf{i}}\hat{\lambda}}_{\mathcal{A}}}\\
\hat{\mathbf{f}}_{\mathcal{A}}\ar[r]^-{{_\mathbf{i}\pi}_{\mathcal{A}}}&{_\mathbf{i}\hat{\mathbf{f}}}_{\mathcal{A}}
}
$$
for any $\mathbf{i}\in\mathfrak{i}$.
Hence there exists an isomorphism ${_\mathfrak{i}\hat{\lambda}}_{\mathcal{A}}:K(_\mathfrak{i}\mathcal{Q})\rightarrow{_\mathfrak{i}\hat{\mathbf{f}}}_{\mathcal{A}}$
such that the following diagram commutes
$$
\xymatrix{
K(\mathcal{Q})\ar[r]^-{{_{\mathfrak{i}}j}^{\ast}}\ar[d]^-{\hat{\lambda}_{\mathcal{A}}}&K(_\mathfrak{i}\mathcal{Q})\ar[d]^-{{_\mathfrak{i}\hat{\lambda}}_{\mathcal{A}}}\\
\hat{\mathbf{f}}_{\mathcal{A}}\ar[r]^-{{_\mathfrak{i}\pi}_{\mathcal{A}}}&{_\mathfrak{i}\hat{\mathbf{f}}}_{\mathcal{A}}
}
$$

\end{proof}

\subsubsection{}

For any $\nu,\nu',\nu''\in\mathbb{N}\mathbf{I}$ such that $\nu=\nu'+\nu''$, fix $\mathbf{V}\in\mathcal{C}'_{\nu}$, $\mathbf{V}'\in\mathcal{C}'_{\nu'}$, $\mathbf{V}''\in\mathcal{C}'_{\nu''}$. Consider the following diagram
\begin{equation}\label{equation:4.2.2}
\xymatrix{{_\mathfrak{i}E}_{\mathbf{V}'}\times {_\mathfrak{i}E}_{\mathbf{V}''}\ar[d]^-{{_{\mathfrak{i}}j}_{\mathbf{V}'}\times{_{\mathfrak{i}}j}_{\mathbf{V}''}}&{_\mathfrak{i}E}'\ar[d]^-{j_1}\ar[l]_-{p_1}\ar[r]^-{p_2}&{_\mathfrak{i}E}''\ar[d]^-{j_2}\ar[r]^-{p_3}&{_\mathfrak{i}E}_{\mathbf{V}}\ar[d]^-{{_{\mathfrak{i}}j}_{\mathbf{V}}}\\
E_{\mathbf{V}'}\times E_{\mathbf{V}''}&E'\ar[l]_-{p_1}\ar[r]^-{p_2}&E''\ar[r]^-{p_3}&E_{\mathbf{V}}}
\end{equation}
where
\begin{enumerate}
  \item[(1)]${_\mathfrak{i}E}'=p^{-1}_1({_\mathfrak{i}E}_{\mathbf{V}'}\times {_\mathfrak{i}E}_{\mathbf{V}''})$;
  \item[(2)]${_\mathfrak{i}E}''=p_2({_\mathfrak{i}E}')$;
  \item[(3)]the restrictions of $p_1$, $p_2$ and $p_3$ are also denoted by $p_1$, $p_2$ and $p_3$ respectively.
\end{enumerate}

For any two complexes $\mathcal{L}'\in\mathcal{D}_{G_{\mathbf{V}'}}({_\mathfrak{i}E}_{\mathbf{V}'})$ and $\mathcal{L}''\in\mathcal{D}_{G_{\mathbf{V}''}}({_\mathfrak{i}E}_{\mathbf{V}''})$, $\mathcal{L}=\mathcal{L}'\ast\mathcal{L}''$ is defined as follows.

Let $\mathcal{L}_1=\mathcal{L}'\otimes\mathcal{L}''$ and $\mathcal{L}_2=p_1^{\ast}\mathcal{L}_1$. Since $p_1$ is smooth with connected fibres and $p_2$ is a $G_{\mathbf{V}'}\times G_{\mathbf{V}''}$-principal bundle, there exists a complex $\mathcal{L}_3$ on ${_\mathfrak{i}E}'$ such that $p_2^{\ast}(\mathcal{L}_3)=\mathcal{L}_2$. $\mathcal{L}$ is defined as $(p_3)_{!}\mathcal{L}_3$.
Note that $\mathcal{L}$ is not a perverse sheaf in general.

The canonical embedding ${_{\mathfrak{i}}j}_{\mathbf{V}}:{_{\mathfrak{i}}E}_{\mathbf{V}}\rightarrow E_{\mathbf{V}}$ also induces a functor
$${_{\mathfrak{i}}j}_{\mathbf{V}!}:\mathcal{D}_{G_{\mathbf{V}}}(_{\mathfrak{i}}E_{\mathbf{V}})\rightarrow \mathcal{D}_{G_{\mathbf{V}}}(E_{\mathbf{V}}).$$

\begin{lemma}\label{lemma:4.6}
It holds that ${_{\mathfrak{i}}j}_{\mathbf{V}!}(\mathcal{L}'\ast\mathcal{L}'')=
{_{\mathfrak{i}}j}_{\mathbf{V}'!}(\mathcal{L}')\ast{_{\mathfrak{i}}j}_{\mathbf{V}''!}(\mathcal{L}'')$.
\end{lemma}

\begin{proof}
Let $\hat{\mathcal{L}}'={_{\mathfrak{i}}j}_{\mathbf{V}'!}\mathcal{L}'$ and $\hat{\mathcal{L}}''={_{\mathfrak{i}}j}_{\mathbf{V}''!}\mathcal{L}''$.
Since
$$
\xymatrix{{_\mathfrak{i}E}_{\mathbf{V}'}\times {_\mathfrak{i}E}_{\mathbf{V}''}\ar[d]^-{{_{\mathfrak{i}}j}_{\mathbf{V}'}\times{_{\mathfrak{i}}j}_{\mathbf{V}''}}&{_\mathfrak{i}E}'\ar[d]^-{j_1}\ar[l]_-{p_1}\\
E_{\mathbf{V}'}\times E_{\mathbf{V}''}&E'\ar[l]_-{p_1}}
$$
is a fiber product, we have $\hat{\mathcal{L}}_2:=p^{\ast}_1(\hat{\mathcal{L}}'\otimes\hat{\mathcal{L}}'')=j_{1!}p^{\ast}_1(\mathcal{L}'\otimes\mathcal{L}'')=j_{1!}\mathcal{L}_2.$

There exists a complex $\hat{\mathcal{L}}_3$ on ${E}'$ such that $p_2^{\ast}(\hat{\mathcal{L}}_3)=\hat{\mathcal{L}}_2$. Since $p^{\ast}_2$ are equivalences of categories,
$\hat{\mathcal{L}}_3=j_{2!}\mathcal{L}_3$.

At last, ${_{\mathfrak{i}}j}_{\mathbf{V}'!}(\mathcal{L}')\ast{_{\mathfrak{i}}j}_{\mathbf{V}''!}(\mathcal{L}'')=\hat{\mathcal{L}}'\ast\hat{\mathcal{L}}''
=p_{3!}\hat{\mathcal{L}}_3=p_{3!}j_{2!}\mathcal{L}_3={_{\mathfrak{i}}j}_{\mathbf{V}!}p_{3!}\mathcal{L}_3={_{\mathfrak{i}}j}_{\mathbf{V}!}(\mathcal{L}'\ast\mathcal{L}'')$.

\end{proof}

Let $K(\mathcal{D}_{G_{\mathbf{V}}}(_{\mathfrak{i}}E_{\mathbf{V}}))$ be the Grothendieck group of $\mathcal{D}_{G_{\mathbf{V}}}(_{\mathfrak{i}}E_{\mathbf{V}})$ and
$K(\mathcal{D}_{G_{\mathbf{V}}}(E_{\mathbf{V}}))$ be the Grothendieck group of $\mathcal{D}_{G_{\mathbf{V}}}(E_{\mathbf{V}})$.
The functor
$${_{\mathfrak{i}}j}_{\mathbf{V}!}:\mathcal{D}_{G_{\mathbf{V}}}(_{\mathfrak{i}}E_{\mathbf{V}})\rightarrow \mathcal{D}_{G_{\mathbf{V}}}(E_{\mathbf{V}})$$
induces a map $${_{\mathfrak{i}}j}_{\mathbf{V}!}:K(\mathcal{D}_{G_{\mathbf{V}}}(_{\mathfrak{i}}E_{\mathbf{V}}))\rightarrow K(\mathcal{D}_{G_{\mathbf{V}}}(E_{\mathbf{V}})).$$

\begin{lemma}\label{lemma:4.7}
It holds that ${_{\mathfrak{i}}j}_{\mathbf{V}!}(K(_\mathfrak{i}\mathcal{Q}_{\mathbf{V}}))\in K(\mathcal{Q}_{\mathbf{V}})$.

\end{lemma}

\begin{proof}

Consider the following diagram
$$
\xymatrix{
K(_\mathfrak{i}\mathcal{Q})\ar[r]^-{{_{\mathfrak{i}}j}_{!}}&\bigoplus_{\nu}K(\mathcal{D}_{G_{\mathbf{V}}}(E_{\mathbf{V}}))
\ar[r]^-{{_{\mathfrak{i}}j}^\ast}&\bigoplus_{\nu}K(\mathcal{D}_{G_{\mathbf{V}}}(_{\mathfrak{i}}E_{\mathbf{V}}))\\
&K(\mathcal{Q})\ar[r]^-{{_{\mathfrak{i}}j}^\ast}\ar[u]^-{}&K(_\mathfrak{i}\mathcal{Q})\ar[u]^-{}\\
{_\mathfrak{i}\hat{\mathbf{f}}}_{\mathcal{A}}\ar[uu]_-{{_\mathfrak{i}\hat{\lambda}}^{-1}_{\mathcal{A}}}\ar[r]^-{}&
\hat{\mathbf{f}}_{\mathcal{A}}\ar[u]_-{{\hat{\lambda}}^{-1}_{\mathcal{A}}}\ar[r]^-{_\mathfrak{i}\pi_{\mathcal{A}}}&{_\mathfrak{i}\hat{\mathbf{f}}}_{\mathcal{A}}\ar[u]_-{{_\mathfrak{i}\hat{\lambda}}^{-1}_{\mathcal{A}}}
}
$$
Since the compositions
$$
\xymatrix{
K(_\mathfrak{i}\mathcal{Q})\ar[r]^-{{_{\mathfrak{i}}j}_{!}}&\bigoplus_{\nu}K(\mathcal{D}_{G_{\mathbf{V}}}(E_{\mathbf{V}}))
\ar[r]^-{{_{\mathfrak{i}}j}^\ast}&\bigoplus_{\nu}K(\mathcal{D}_{G_{\mathbf{V}}}(_{\mathfrak{i}}E_{\mathbf{V}}))}
$$
and
$$
\xymatrix{
{_\mathfrak{i}\hat{\mathbf{f}}}_{\mathcal{A}}\ar[r]^-{}&
\hat{\mathbf{f}}_{\mathcal{A}}\ar[r]^-{_\mathfrak{i}\pi_{\mathcal{A}}}&
{_\mathfrak{i}\hat{\mathbf{f}}}_{\mathcal{A}}
}
$$
are identifies, the following diagram commutes
\begin{equation}\label{equation:4.2.1}
\xymatrix{
K(_\mathfrak{i}\mathcal{Q})\ar[r]^-{{_{\mathfrak{i}}j}_{!}}&\bigoplus_{\nu}K(\mathcal{D}_{G_{\mathbf{V}}}(E_{\mathbf{V}}))
\ar[r]^-{{_{\mathfrak{i}}j}^\ast}&\bigoplus_{\nu}K(\mathcal{D}_{G_{\mathbf{V}}}(_{\mathfrak{i}}E_{\mathbf{V}}))\\
&&K(_\mathfrak{i}\mathcal{Q})\ar[u]^-{}\\
{_\mathfrak{i}\hat{\mathbf{f}}}_{\mathcal{A}}\ar[uu]_-{{_\mathfrak{i}\hat{\lambda}}^{-1}_{\mathcal{A}}}\ar[r]^-{}&
\hat{\mathbf{f}}_{\mathcal{A}}\ar[r]^-{_\mathfrak{i}\pi_{\mathcal{A}}}&{_\mathfrak{i}\hat{\mathbf{f}}}_{\mathcal{A}}\ar[u]_-{{_\mathfrak{i}\hat{\lambda}}^{-1}_{\mathcal{A}}}
}
\end{equation}

For any homogeneous $x\in{_\mathfrak{i}\hat{\mathbf{f}}}_{\mathcal{A}}$,
choose $\mathcal{L}\in{_\mathfrak{i}\mathcal{Q}}_{\mathbf{V}}$ such that $[\mathcal{L}]={_\mathfrak{i}\hat{\lambda}}^{-1}_{\mathcal{A}}(x)$.
Let $\mathcal{L}_1={_{\mathfrak{i}}j}_{!}\mathcal{L}$. It is clear that supp$(\mathcal{L}_1)\in{_{\mathfrak{i}}E_{\mathbf{V}}}$.

The subalgebra ${_\mathfrak{i}\hat{\mathbf{f}}}$ of $\hat{\mathbf{f}}$ is generated by $f(\mathbf{i},\mathbf{j};m)$ for all $\mathbf{i}\in\mathfrak{i}$, $\mathbf{j}\not\in\mathfrak{i}$ and $m\leq -a_{\mathbf{i}\mathbf{j}}$.
Let $\nu^{(m)}=m\mathbf{i}+\mathbf{j}\in\mathbb{N}\mathbf{I}$. Fix an object $\mathbf{V}^{(m)}\in\mathcal{C}'$ such that $\underline{\dim}\mathbf{V}^{(m)}=\nu^{(m)}$.
Denote by $\mathbf{1}_{_\mathbf{i}E_{\mathbf{V}^{(m)}}}\in\mathcal{D}_{G_{\mathbf{V}^{(m)}}}(_\mathbf{i}E_{\mathbf{V}^{(m)}})$ the constant sheaf on $_\mathbf{i}E_{\mathbf{V}^{(m)}}$.
Define
$$\mathcal{E}^{(m)}=j_{\mathbf{V}^{(m)}!}(v^{-mN}\mathbf{1}_{_\mathbf{i}E_{\mathbf{V}^{(m)}}})\in\mathcal{D}_{G_{\mathbf{V}^{(m)}}}(E_{\mathbf{V}^{(m)}}).$$
In \cite{Xiao_Zhao_Geometric_realizations_of_Lusztig's_symmetries}, it was proved that $[\mathcal{E}^{(m)}]={\hat{\lambda}}^{-1}_{\mathcal{A}}(f(\mathbf{i},\mathbf{j};m))$.
Since supp$(\mathcal{E}^{(m)})\in{_{\mathfrak{i}}E_{\mathbf{V}^{(m)}}}$,
there exists $\mathcal{L}_2$ such that $[\mathcal{L}_2]={\hat{\lambda}}^{-1}_{\mathcal{A}}(x)$ and supp$(\mathcal{L}_2)\in{_{\mathfrak{i}}E_{\mathbf{V}}}$.

The Diagram (\ref{equation:4.2.1}) implies $[{_{\mathfrak{i}}j}^\ast\mathcal{L}_1]=[{_{\mathfrak{i}}j}^\ast\mathcal{L}_2]$. Hence $[\mathcal{L}_1]=[\mathcal{L}_2]$.
That is, the following diagram commutes
$$
\xymatrix{
K(_\mathfrak{i}\mathcal{Q})\ar[r]^-{{_{\mathfrak{i}}j}_{!}}&\bigoplus_{\nu}K(\mathcal{D}_{G_{\mathbf{V}}}(E_{\mathbf{V}}))
\\
&K(\mathcal{Q})\ar[u]^-{}\\
{_\mathfrak{i}\hat{\mathbf{f}}}_{\mathcal{A}}\ar[uu]_-{{_\mathfrak{i}\hat{\lambda}}^{-1}_{\mathcal{A}}}\ar[r]&
\hat{\mathbf{f}}_{\mathcal{A}}\ar[u]_-{{\hat{\lambda}}^{-1}_{\mathcal{A}}}
}
$$
Hence ${_{\mathfrak{i}}j}_{\mathbf{V}!}(K(_\mathfrak{i}\mathcal{Q}_{\mathbf{V}}))\in K(\mathcal{Q}_{\mathbf{V}})$.

\end{proof}
%


\begin{lemma}\label{lemma:4.8}
For any $\mathcal{L}'\in{_\mathfrak{i}\mathcal{Q}}_{\mathbf{V}'}$ and $\mathcal{L}''\in{_\mathfrak{i}\mathcal{Q}}_{\mathbf{V}''}$, $[\mathcal{L}'\ast\mathcal{L}'']\in K({_\mathfrak{i}}\mathcal{Q}_{\mathbf{V}})$.
\end{lemma}

\begin{proof}
Let $\hat{\mathcal{L}}'={_{\mathfrak{i}}j}_{\mathbf{V}'!}(\mathcal{L}')$ and $\hat{\mathcal{L}}''={_{\mathfrak{i}}j}_{\mathbf{V}''!}(\mathcal{L}'')$.
Lemma \ref{lemma:4.6} implies
$\mathcal{L}'\ast\mathcal{L}''={_{\mathfrak{i}}j}^\ast_{\mathbf{V}}{_{\mathfrak{i}}j}_{\mathbf{V}!}(\mathcal{L}'\ast\mathcal{L}'')=
{_{\mathfrak{i}}j}^\ast_{\mathbf{V}}(\hat{\mathcal{L}}'\ast\hat{\mathcal{L}}'')$. By Lemma \ref{lemma:4.7}, $\textrm{$[\hat{\mathcal{L}}']$ and $[\hat{\mathcal{L}}'']$}\in K(\mathcal{Q}_{\mathbf{V}})$. Hence $[{_{\mathfrak{i}}j}_{\mathbf{V}!}(\mathcal{L}'\ast\mathcal{L}'')]\in K(\mathcal{Q}_{\mathbf{V}})$.
So $[\mathcal{L}'\ast\mathcal{L}'']\in K({_\mathfrak{i}}\mathcal{Q}_{\mathbf{V}})$.

\end{proof}

Hence, we get
an associative $\mathcal{A}$-bilinear multiplication
\begin{eqnarray*}
\circledast:K({_\mathfrak{i}\mathcal{Q}}_{\mathbf{V}'})\times K({_\mathfrak{i}\mathcal{Q}}_{\mathbf{V}''})&\rightarrow&K({_\mathfrak{i}\mathcal{Q}}_{\mathbf{V}})\\
([\mathcal{L}']\,,\,[\mathcal{L}''])&\mapsto&[\mathcal{L}']\circledast[\mathcal{L}'']=[\mathcal{L}'\circledast\mathcal{L}'']
\end{eqnarray*}
where $\mathcal{L}'\circledast\mathcal{L}''=(\mathcal{L}'\ast\mathcal{L}'')[m_{\nu'\nu''}](\frac{m_{\nu'\nu''}}{2})$.
Then $K({_\mathfrak{i}\mathcal{Q}})$ becomes an associative $\mathcal{A}$-algebra and the set $\{[\mathcal{L}]\,\,|\,\,\mathcal{L}\in{_{\mathfrak{i}}}\mathcal{P}_{\mathbf{V}}\}$ is a basis of $K({_\mathfrak{i}\mathcal{Q}}_{\mathbf{V}})$.

%

\begin{proposition}\label{proposition:4.9}
We have the following commutative diagram
$$
\xymatrix{
K(_\mathfrak{i}\mathcal{Q})\ar[r]^-{{_{\mathfrak{i}}j}_{!}}\ar[d]^-{{_\mathfrak{i}\hat{\lambda}}_{\mathcal{A}}}&K(\mathcal{Q})\ar[d]^-{{\hat{\lambda}}_{\mathcal{A}}}\\
{_\mathfrak{i}\hat{\mathbf{f}}}_{\mathcal{A}}\ar[r]^-{}&\hat{\mathbf{f}}_{\mathcal{A}}
}
$$
Moreover, ${_\mathfrak{i}\hat{\lambda}}_{\mathcal{A}}:K(_\mathfrak{i}\mathcal{Q})\rightarrow{_\mathfrak{i}\hat{\mathbf{f}}}_{\mathcal{A}}$ is an isomorphism of algebras.

\end{proposition}

\begin{proof}

By the proof of Lemma \ref{lemma:4.7}, we have the following commutative diagram
$$
\xymatrix{
K(_\mathfrak{i}\mathcal{Q})\ar[r]^-{{_{\mathfrak{i}}j}_{!}}\ar[d]^-{{_\mathfrak{i}\hat{\lambda}}_{\mathcal{A}}}&K(\mathcal{Q})\ar[r]^-{{_{\mathfrak{i}}j}^\ast}\ar[d]^-{\hat{\lambda}_{\mathcal{A}}}&K(_\mathfrak{i}\mathcal{Q})\ar[d]^-{{_\mathfrak{i}\hat{\lambda}}_{\mathcal{A}}}\\
{_\mathfrak{i}\hat{\mathbf{f}}}_{\mathcal{A}}\ar[r]^-{}&\hat{\mathbf{f}}_{\mathcal{A}}\ar[r]^-{_\mathfrak{i}\pi_{\mathcal{A}}}&{_\mathfrak{i}\hat{\mathbf{f}}}_{\mathcal{A}}
}
$$

Lemma \ref{lemma:4.6} implies ${_{\mathfrak{i}}j}_{!}:K(_\mathfrak{i}\mathcal{Q})\rightarrow K(\mathcal{Q})$ is a monomorphism of algebras.
Since $\hat{\lambda}_{\mathcal{A}}:K(\mathcal{Q})\rightarrow\hat{\mathbf{f}}_{\mathcal{A}}$ is an isomorphism of algebras,  ${_\mathfrak{i}\hat{\lambda}}_{\mathcal{A}}:K(_\mathfrak{i}\mathcal{Q})\rightarrow{_\mathfrak{i}\hat{\mathbf{f}}}_{\mathcal{A}}$ is also an isomorphism of algebras.


\end{proof}

\subsection{Geometric realization of ${_i\mathbf{f}}$}\label{subsection:4.3}

\subsubsection{}

Let $\tilde{Q}=(Q,a)$ be a quiver with automorphism, where $Q=(\mathbf{I},H,s,t)$. Fix $i\in I=\mathbf{I}^a$ and assume that $\mathbf{i}$ is a sink for any $\mathbf{i}\in i$.

For any $\nu\in\mathbb{N}\mathbf{I}^a$ and $\mathbf{V}\in\tilde{\mathcal{C}}_{\nu}$, $\mathbf{V}$ can be viewed as an object in $\mathcal{C}'_{\nu}$.
Hence $_i{E_{\mathbf{V}}}$ is defined in Section \ref{subsection:4.2}.
The morphism $a:E_{\mathbf{V}}\rightarrow E_{\mathbf{V}}$ satisfies that $a({_iE_{\mathbf{V}}})={_iE_{\mathbf{V}}}$.
Hence we have a functor $a^{\ast}:\mathcal{D}_{G_{\mathbf{V}}}({_iE_{\mathbf{V}}})\rightarrow\mathcal{D}_{G_{\mathbf{V}}}({_iE_{\mathbf{V}}})$.

\begin{lemma}\label{lemma:4.9}
It holds that $a^{\ast}({_i\mathcal{Q}_{\mathbf{V}}})={_i\mathcal{Q}_{\mathbf{V}}}$.
\end{lemma}

\begin{proof}
Note that $a\circ{_ij}_{\mathbf{V}}={_ij}_{\mathbf{V}}\circ a$. Hence $a^{\ast}{_ij}^{\ast}_{\mathbf{V}}={_ij}^{\ast}_{\mathbf{V}}a^{\ast}$.
Since $a^{\ast}(\mathcal{Q}_{\mathbf{V}})=\mathcal{Q}_{\mathbf{V}}$ and ${_ij}^{\ast}_{\mathbf{V}}(\mathcal{Q}_{\mathbf{V}})={_i\mathcal{Q}_{\mathbf{V}}}$, $a^{\ast}({_i\mathcal{Q}_{\mathbf{V}}})={_i\mathcal{Q}_{\mathbf{V}}}$.

\end{proof}

Similarly to ${\tilde{\mathcal{Q}}_{\mathbf{V}}}$, we can get a category ${_i\tilde{\mathcal{Q}}_{\mathbf{V}}}$.
The objects in ${_i\tilde{\mathcal{Q}}_{\mathbf{V}}}$ are pairs $(\mathcal{L},\phi)$,
where $\mathcal{L}\in{_i\mathcal{Q}_{\mathbf{V}}}$ and $\phi:a^{\ast}\mathcal{L}\rightarrow\mathcal{L}$ is an isomorphism such that
$$a^{\ast\mathbf{n}}\mathcal{L}\rightarrow
a^{\ast(\mathbf{n}-1)}\mathcal{L}\rightarrow
\ldots
\rightarrow
a^{\ast}\mathcal{L}\rightarrow\mathcal{L}$$
is the identity map of $\mathcal{L}$.
A morphism in $\textrm{Hom}_{{_i\tilde{\mathcal{Q}}_{\mathbf{V}}}}((\mathcal{L},\phi),(\mathcal{L}',\phi'))$ is
a morphism $f\in\textrm{Hom}_{{_i\mathcal{Q}}_{\mathbf{V}}}(\mathcal{L},\mathcal{L}')$
such that
$$
\xymatrix{
a^{\ast}\mathcal{L}\ar[r]^-{\phi}\ar[d]^-{a^{\ast}f}&\mathcal{L}\ar[d]^-{f}\\
a^{\ast}\mathcal{L}'\ar[r]^-{\phi'}&\mathcal{L}'
}
$$
For any $(\mathcal{L},\phi)\in\tilde{\mathcal{Q}}_{\mathbf{V}}$, the map ${_i\phi}={_ij}^{\ast}_{\mathbf{V}}\phi:a^{\ast}{_ij}^{\ast}_{\mathbf{V}}\mathcal{L}={_ij}^{\ast}_{\mathbf{V}}a^{\ast}\mathcal{L}\rightarrow{_ij}^{\ast}_{\mathbf{V}}\mathcal{L}$
is also an isomorphism.
Hence we get a functor ${_ij}^{\ast}_{\mathbf{V}}:\tilde{\mathcal{Q}}_{\mathbf{V}}\rightarrow{_i\tilde{\mathcal{Q}}_{\mathbf{V}}}$.
Similarly, $K(_i\tilde{\mathcal{Q}}_{\mathbf{V}})$ has a natural $\mathcal{O}'$-module structure.

For $\nu,\nu',\nu''\in\mathbb{N}I=\mathbb{N}\mathbf{I}^a$ such that $\nu=\nu'+\nu''$, fix $\mathbf{V}\in\tilde{\mathcal{C}}_{\nu}$, $\mathbf{V}''\in\tilde{\mathcal{C}}_{\nu''}$ $\mathbf{V}'\in\tilde{\mathcal{C}}_{\nu'}$, $\mathbf{V}''\in\tilde{\mathcal{C}}_{\nu''}$.
Similarly to Lemma \ref{lemma:3.5}, 
%
the induction functor
$$\ast:\mathcal{D}_{G_{\mathbf{V}'}}({_{i}E}_{\mathbf{V}'})\times\mathcal{D}_{G_{\mathbf{V}''}}({_{i}E}_{\mathbf{V}''})
\rightarrow\mathcal{D}_{G_{\mathbf{V}}}(_{i}E_{\mathbf{V}})$$ is compatible with $a^\ast$.
By Lemma \ref{lemma:4.8}, we have
an associative $\mathcal{O}'$-bilinear multiplication
\begin{eqnarray*}
\circledast:K({_i\tilde{\mathcal{Q}}}_{\mathbf{V}'})\times K({_i\tilde{\mathcal{Q}}}_{\mathbf{V}''})&\rightarrow&K({_i\tilde{\mathcal{Q}}}_{\mathbf{V}})\\
([\mathcal{L}',\phi']\,,\,[\mathcal{L}'',\phi''])&\mapsto&[\mathcal{L},\phi]
\end{eqnarray*}
where $(\mathcal{L},\phi)=(\mathcal{L}',\phi')\circledast(\mathcal{L}'',\phi'')=((\mathcal{L}',\phi')\ast(\mathcal{L}'',\phi''))[m_{\nu'\nu''}](\frac{m_{\nu'\nu''}}{2})$.
Then $K({_i\tilde{\mathcal{Q}}})$ becomes an associative $\mathcal{O}'$-algebra.

\subsubsection{}

Fix a nonzero element $\nu\in\mathbb{N}\mathbf{I}^a$
and $\mathbf{V}\in\tilde{\mathcal{C}}_{\nu}$.
For any element $\mathbf{y}\in Y^a_{\nu}$,
the automorphism $a:{_iF}_{\mathbf{y}}\rightarrow {_iF}_{\mathbf{y}}$ is defined as
$$a(\phi)=(\mathbf{V}=a(\mathbf{V}^k)\supset a(\mathbf{V}^{k-1})\supset\dots\supset a(\mathbf{V}^0)=0)$$
for any $$\phi=(\mathbf{V}=\mathbf{V}^k\supset\mathbf{V}^{k-1}\supset\dots\supset\mathbf{V}^0=0)\in {_iF}_{\mathbf{y}}.$$
We also have an automorphism $a:{_i\tilde{F}}_{\mathbf{y}}\rightarrow {_i\tilde{F}}_{\mathbf{y}}$, defined as $a((x,\phi))=(a(x),a(\phi))$ for any $(x,\phi)\in{_i\tilde{F}}_{\mathbf{y}}$.

Since $a^{\ast}(\mathbf{1}_{{_i\tilde{F}}_{\mathbf{y}}})\cong\mathbf{1}_{{_i\tilde{F}}_{\mathbf{y}}}$, there exists an isomorphism
$${_i\phi}_0:a^{\ast}{_i\mathcal{L}}_{\mathbf{y}}=a^{\ast}(\pi_{\mathbf{y}})_!(\mathbf{1}_{{_i\tilde{F}}_{\mathbf{y}}})[d_{\mathbf{y}}]=
(\pi_{\mathbf{y}})_!a^{\ast}(\mathbf{1}_{\tilde{F}_{\mathbf{y}}})[d_{\mathbf{y}}]
\cong(\pi_{\mathbf{y}})_!(\mathbf{1}_{{_i\tilde{F}}_{\mathbf{y}}})[d_{\mathbf{y}}]={_i\mathcal{L}}_{\mathbf{y}}.$$
Then $({_i\mathcal{L}}_{\mathbf{y}},{_i\phi}_0)$ is an object in ${_i\tilde{\mathcal{Q}}}_{\mathbf{V}}$.

Let ${_{i}\mathbf{k}}_{\nu}$ be the $\mathcal{A}$-submodule of $K({_i\tilde{\mathcal{Q}}}_{\mathbf{V}})$ spanned by $({_i\mathcal{L}}_{\mathbf{y}},{_i\phi}_0)$ for all $\mathbf{y}\in Y^a_{\nu}$. Let ${_{i}\mathbf{k}}=\bigoplus_{\nu\in\mathbb{N}I}{_{i}\mathbf{k}}_{\nu}$, which is also a subalgebra of $K({_i\tilde{\mathcal{Q}}})$.

\subsubsection{}

Since $({_i\mathcal{L}}_{\mathbf{y}},{_i\phi}_0)={_ij}^{\ast}_{\mathbf{V}}(\mathcal{L}_{\mathbf{y}},\phi_0)$,
the functor ${_ij}^{\ast}_{\mathbf{V}}:\tilde{\mathcal{Q}}_{\mathbf{V}}\rightarrow{_i\tilde{\mathcal{Q}}_{\mathbf{V}}}$ induces a map ${_ij}^{\ast}:\mathbf{k}\rightarrow{_i\mathbf{k}}$.

\begin{theorem}\label{theorem:4.11}
There is an isomorphism of vector spaces
$${_i\lambda}_{\mathcal{A}}:{_i\mathbf{k}}\rightarrow{_i\mathbf{f}}_{\mathcal{A}}$$
such that the following diagram commutes
$$
\xymatrix{
\mathbf{k}\ar[r]^-{{_ij}^{\ast}}\ar[d]^-{\lambda_{\mathcal{A}}}&{_i\mathbf{k}}\ar[d]^-{{_i\lambda}_{\mathcal{A}}}\\
\mathbf{f}_{\mathcal{A}}\ar[r]^-{_i\pi_\mathcal{A}}&{_i\mathbf{f}}_{\mathcal{A}}
}
$$
\end{theorem}

For the proof of Theorem \ref{theorem:4.11}, we need the following lemmas.

\begin{lemma}\label{lemma:4.12}
There exists an embedding ${_i\delta}:{_i\mathbf{f}}\rightarrow{_i\hat{\mathbf{f}}}$ such that the following diagram commutes
$$
\xymatrix{
\mathbf{f}_{\mathcal{A}}\ar[r]^-{_i\pi_\mathcal{A}}\ar[d]^-{\delta}&{_{i}\mathbf{f}}_{\mathcal{A}}\ar[d]^-{{_i\delta}}\\
\hat{\mathbf{f}}_{\mathcal{A}}\ar[r]^-{_i\pi_\mathcal{A}}&{_{i}\hat{\mathbf{f}}}_{\mathcal{A}}
}
$$
\end{lemma}

\begin{proof}
Consider the following diagram
\begin{equation}\label{equation:4.3.1}
\xymatrix{
0\ar[r]&\theta_i\mathbf{f}_{\mathcal{A}}\ar[r]\ar[d]&\mathbf{f}_{\mathcal{A}}\ar[r]^-{_i\pi_{\mathcal{A}}}\ar[d]^-{\delta}&{_{i}\mathbf{f}}_{\mathcal{A}}\ar@{.>}[d]^-{_i\delta}\ar[r]&0\\
0\ar[r]&\sum_{\mathbf{i}\in i}\theta_\mathbf{i}\hat{\mathbf{f}}_{\mathcal{A}}\ar[r]&\hat{\mathbf{f}}_{\mathcal{A}}\ar[r]^-{_i\pi_{\mathcal{A}}}&{_{i}\hat{\mathbf{f}}}_{\mathcal{A}}\ar[r]&0
}
\end{equation}

Note that $\theta_i\mathbf{f}_{\mathcal{A}}\cap\mathbf{B}$ is a basis of $\theta_i\mathbf{f}_{\mathcal{A}}$, $\sum_{\mathbf{i}\in i}\theta_\mathbf{i}\hat{\mathbf{f}}_{\mathcal{A}}\cap\hat{\mathbf{B}}$ is a basis of $\sum_{\mathbf{i}\in i}\theta_\mathbf{i}\hat{\mathbf{f}}_{\mathcal{A}}$ and $\delta(\mathbf{B})=\hat{\mathbf{B}}^a$.
For any $[B,\phi]\in\theta_i\mathbf{f}_{\mathcal{A}}\cap\mathbf{B}$, $B\in\mathcal{P}_{i,\gamma_i}\cap\hat{\mathbf{B}}^a$, where $\gamma_i=\sum_{\mathbf{i}\in i}\mathbf{i}$.
Since $\mathcal{P}_{i,\gamma_i}\in\mathcal{P}_{\mathbf{i},\mathbf{i}}$, $B\in\mathcal{P}_{\mathbf{i},\mathbf{i}}$ for any $\mathbf{i}\in i$. Hence $\delta([B,\phi])=[B]\in\sum_{\mathbf{i}\in i}\theta_\mathbf{i}\hat{\mathbf{f}}_{\mathcal{A}}\cap\hat{\mathbf{B}}^a$. So there exists an map ${_i\delta}:{_i\mathbf{f}}\rightarrow{_i\hat{\mathbf{f}}}$ such that the Diagram (\ref{equation:4.3.1}) commutes.

On the other hand, for any $[B]\in\sum_{\mathbf{i}\in i}\theta_\mathbf{i}\hat{\mathbf{f}}_{\mathcal{A}}\cap\hat{\mathbf{B}}^a$, $[B]\in\theta_\mathbf{i}\hat{\mathbf{f}}_{\mathcal{A}}\cap\hat{\mathbf{B}}^a$ for some $\mathbf{i}\in i$. Hence
$B\in\mathcal{P}_{\mathbf{i},\mathbf{i}}\cap\hat{\mathbf{B}}^a$. By Lemma 12.5.1 in \cite{Lusztig_Introduction_to_quantum_groups}, $B\in\mathcal{P}_{i,\gamma_i}$. So $[B]\in\delta(\theta_i\mathbf{f}_{\mathcal{A}}\cap\mathbf{B})$.
Hence ${_i\delta}:{_i\mathbf{f}}\rightarrow{_i\hat{\mathbf{f}}}$ is an embedding.

\end{proof}

\begin{lemma}\label{lemma:4.13}
There exists an embedding ${_i\tilde{\delta}}:{_i\mathbf{k}}\rightarrow K(_i\mathcal{Q})$ such that the following diagram commutes
$$
\xymatrix{
\mathbf{k}\ar[r]^-{{_ij}^{\ast}}\ar[d]^-{\tilde{\delta}}&{_i\mathbf{k}}\ar[d]^-{_i\tilde{\delta}}\\
K(\mathcal{Q})\ar[r]^-{{_ij}^{\ast}}&K(_i\mathcal{Q})
}
$$
\end{lemma}
\begin{proof}
For any $[_i\mathcal{L},_i\phi]\in{_i\mathbf{k}}$, where $_i\mathcal{L}$ is a simple perverse sheaf,
there exist $[\mathcal{L},\phi]\in\mathbf{k}$ such that ${_ij}^{\ast}(\mathcal{L},\phi)=(_i\mathcal{L},_i\phi)$.
Define ${_i\tilde{\delta}}([_i\mathcal{L},_i\phi])={_ij}^{\ast}\tilde{\delta}[\mathcal{L},\phi]$. Since the definition of ${_i\tilde{\delta}}$ is independent of the choice of $[\mathcal{L},\phi]$, we get the desired embedding ${_i\tilde{\delta}}:{_i\mathbf{k}}\rightarrow K(_i\mathcal{Q})$.

\end{proof}

\begin{proof}[Proof of Theorem \ref{theorem:4.11}]
Consider the following diagram
$$
\xymatrix{
\mathbf{k}\ar[rrr]^-{{_ij}^{\ast}}\ar[rd]^-{{\lambda}_{\mathcal{A}}}\ar[ddd]^-{\tilde{\delta}}&&&{_i\mathbf{k}}\ar[ddd]^-{_i\tilde{\delta}}\ar@{.>}[ld]_-{{_i\lambda}_{\mathcal{A}}}\\
&\mathbf{f}_{\mathcal{A}}\ar[r]^-{_i\pi_{\mathcal{A}}}\ar[d]^-{\delta}&{_{i}\mathbf{f}}_{\mathcal{A}}\ar[d]^-{_i\delta}&\\
&\hat{\mathbf{f}}_{\mathcal{A}}\ar[r]^-{_i\pi_{\mathcal{A}}}&{_{i}\hat{\mathbf{f}}}_{\mathcal{A}}&\\
K(\mathcal{Q})\ar[rrr]^-{{_ij}^{\ast}}\ar[ru]^-{{\hat{\lambda}}_{\mathcal{A}}}&&&K(_i\mathcal{Q})\ar[lu]_-{{_i\hat{\lambda}}_{\mathcal{A}}}
}
$$
Lemma \ref{lemma:4.12}, \ref{lemma:4.13} and Proposition \ref{proposition:4.5} imply that
there exists a unique $\mathcal{A}$-linear isomorphism
$${_i\lambda}_{\mathcal{A}}:{_i\mathbf{k}}\rightarrow{_i\mathbf{f}}_{\mathcal{A}}$$
such that the following diagram commutes
$$
\xymatrix{
\mathbf{k}\ar[r]^-{{_ij}^{\ast}}\ar[d]^-{\lambda_{\mathcal{A}}}&{_i\mathbf{k}}\ar[d]^-{{_i\lambda}_{\mathcal{A}}}\\
\mathbf{f}_{\mathcal{A}}\ar[r]^-{_i\pi_\mathcal{A}}&{_i\mathbf{f}}_{\mathcal{A}}
}
$$

\end{proof}

\subsubsection{}

In Proposition \ref{proposition:4.9}, we have defined a map ${_{\mathfrak{i}}j}_{\mathbf{V}!}:K(_\mathfrak{i}\mathcal{Q})\rightarrow K(\mathcal{Q})$.
Since $a\circ{_ij}_{\mathbf{V}}={_ij}_{\mathbf{V}}\circ a$, we have $a^{\ast}{_ij}_{\mathbf{V}!}={_ij}_{\mathbf{V}!}a^{\ast}$.
For any isomorphism $\phi:a^\ast\mathcal{L}\rightarrow\mathcal{L}$, the map $\phi'={_ij}_{\mathbf{V}!}\phi:a^{\ast}{_ij}_{\mathbf{V}!}\mathcal{L}={_ij}_{\mathbf{V}!}a^{\ast}\mathcal{L}\rightarrow{_ij}_{\mathbf{V}!}\mathcal{L}$
is also an isomorphism.
Hence we get a map ${_{\mathfrak{i}}j}_!:K(_\mathfrak{i}\tilde{\mathcal{Q}})\rightarrow K(\tilde{\mathcal{Q}})$.

Consider Diagram (\ref{equation:4.2.2}). Since $a$ commutes with $p_1,p_2,p_3$ and $j_1,j_2$, Lemma \ref{lemma:4.6} implies that ${_{\mathfrak{i}}j}_!:K(_\mathfrak{i}\tilde{\mathcal{Q}})\rightarrow K(\tilde{\mathcal{Q}})$ is a homomorphism of algebras.

\begin{theorem}\label{theorem:4.14}
It holds that ${_ij}_{!}({_i\mathbf{k}})\subset\mathbf{k}$
and we have the following commutative diagram
$$
\xymatrix{
{_i\mathbf{k}}\ar[d]^-{{_i\lambda}_{\mathcal{A}}}\ar[r]^-{{_ij}_{!}}&\mathbf{k}\ar[d]^-{\lambda_{\mathcal{A}}}\\
{_i\mathbf{f}}_{\mathcal{A}}\ar[r]&\mathbf{f}_{\mathcal{A}}
}
$$
Moreover, ${_i\lambda}_{\mathcal{A}}:{_i\mathbf{k}}\rightarrow{_i\mathbf{f}}_{\mathcal{A}}$
is an isomorphism of algebras.
\end{theorem}

\begin{proof}
Consider the following diagram
$$
\xymatrix{
{_i\mathbf{k}}\ar[r]^-{{_ij}_{!}}&K(\tilde{\mathcal{Q}})
\ar[r]^-{{_ij}^\ast}&K(_{i}\tilde{\mathcal{Q}})\\
&\mathbf{k}\ar[r]^-{{_ij}^\ast}\ar[u]^-{}&{_i\mathbf{k}}\ar[u]^-{}\\
{_{i}\mathbf{f}}_{\mathcal{A}}\ar[uu]^-{{_i\lambda}^{-1}_{\mathcal{A}}}\ar[r]^-{}&\mathbf{f}_{\mathcal{A}}\ar[u]^-{{\lambda}^{-1}_{\mathcal{A}}}\ar[r]^-{_i\pi_{\mathcal{A}}}&{_{i}\mathbf{f}}_{\mathcal{A}}\ar[u]^-{{_i\lambda}^{-1}_{\mathcal{A}}}
}
$$
Since the compositions
$$
\xymatrix{
{_i\mathbf{k}}\ar[r]^-{{_ij}_{!}}&K(\tilde{\mathcal{Q}})
\ar[r]^-{{_ij}^\ast}&K(_{i}\tilde{\mathcal{Q}})}
$$
and
$$
\xymatrix{
{_{i}\mathbf{f}}_{\mathcal{A}}\ar[r]^-{}&\mathbf{f}_{\mathcal{A}}\ar[r]^-{_i\pi_{\mathcal{A}}}&{_{i}\mathbf{f}}_{\mathcal{A}}
}
$$
are identifies, the following diagram commutes
\begin{equation}\label{equation:4.3.3}
\xymatrix{
{_i\mathbf{k}}\ar[r]^-{{_ij}_{!}}&K(\tilde{\mathcal{Q}})
\ar[r]^-{{_ij}^\ast}&K(_{i}\tilde{\mathcal{Q}})\\
&&{_i\mathbf{k}}\ar[u]^-{}\\
{_{i}\mathbf{f}}_{\mathcal{A}}\ar[uu]^-{{_i\lambda}^{-1}_{\mathcal{A}}}\ar[r]^-{}&\mathbf{f}_{\mathcal{A}}\ar[r]^-{_i\pi_{\mathcal{A}}}&{_{i}\mathbf{f}}_{\mathcal{A}}\ar[u]^-{{_i\lambda}^{-1}_{\mathcal{A}}}
}
\end{equation}

For any homogeneous $x\in{_{i}\mathbf{f}}_{\mathcal{A}}$,
choose $(\mathcal{L},\phi)\in\mathcal{Q}_{\mathbf{V}}$ such that $[\mathcal{L},\phi]={_{i}\lambda}^{-1}_{\mathcal{A}}(x)$.
Let $(\mathcal{L}_1,\phi_1)={_{i}j}_{!}(\mathcal{L},\phi)$. It is clear that supp$(\mathcal{L}_1)\in{_iE}_{\mathbf{V}}$.

The subalgebra ${_{i}\mathbf{f}}$ of $\mathbf{f}$ is generated by $f(i,j;m)$ for all $j\neq{i}$ and $m\leq -a_{ij}$.
Let $\nu^{(m)}=m\gamma_i+\gamma_j\in\mathbb{N}\mathbf{I}^a$. Fix an object $\mathbf{V}^{(m)}\in\tilde{\mathcal{C}}$ such that $\underline{\dim}\mathbf{V}^{(m)}=\nu^{(m)}$.
Denote by $\mathbf{1}_{_{i}E_{\mathbf{V}^{(m)}}}\in\mathcal{D}_{G_{\mathbf{V}^{(m)}}}(_{i}E_{\mathbf{V}^{(m)}})$ the constant sheaf on $_{i}E_{\mathbf{V}^{(m)}}$.
Define
$$\mathcal{E}^{(m)}=j_{\mathbf{V}^{(m)}!}(v^{-mN}\mathbf{1}_{_iE_{\mathbf{V}^{(m)}}})\in\mathcal{D}_{G_{\mathbf{V}^{(m)}}}(E_{\mathbf{V}^{(m)}}).$$
In Section \ref{subsection:5.2}, it will be proved that $[\mathcal{E}^{(m)},\mathrm{id}]={\hat{\lambda}}^{-1}_{\mathcal{A}}(f(i,j;m))$.
Since supp$(\mathcal{E}^{(m)})\in{_{{i}}E_{\mathbf{V}^{(m)}}}$,
there exists $(\mathcal{L}_2,\phi_2)$ such that $[\mathcal{L}_2,\phi_2]={\hat{\lambda}}^{-1}_{\mathcal{A}}(x)$ and supp$(\mathcal{L}_2)\in{_{{i}}E_{\mathbf{V}}}$.

The Diagram (\ref{equation:4.3.3}) implies $[{_{{i}}j}^\ast\mathcal{L}_1,{_{{i}}j}^\ast\phi_1]=[{_{{i}}j}^\ast\mathcal{L}_2,{_{{i}}j}^\ast\phi_2]$. Hence $[\mathcal{L}_1,\phi_1]=[\mathcal{L}_2,\phi_2]$.
That is the following diagram commutes
$$
\xymatrix{
{_i\mathbf{k}}\ar[r]^-{{_ij}_{!}}&K(\tilde{\mathcal{Q}})
\\
&\mathbf{k}\ar[u]^-{}\\
{_{i}\mathbf{f}}_{\mathcal{A}}\ar[uu]^-{{_i\lambda}^{-1}_{\mathcal{A}}}\ar[r]^-{}&\mathbf{f}_{\mathcal{A}}\ar[u]^-{{\lambda}^{-1}_{\mathcal{A}}}
}
$$
Hence ${_ij}_{!}({_i\mathbf{k}})\subset\mathbf{k}$ and we have the following commutative diagram
$$
\xymatrix{
{_i\mathbf{k}}\ar[d]^-{{_i\lambda}_{\mathcal{A}}}\ar[r]^-{{_ij}_{!}}&\mathbf{k}\ar[r]^-{{_ij}^\ast}\ar[d]^-{\lambda_{\mathcal{A}}}&{_i\mathbf{k}}\ar[d]^-{{_i\lambda}_{\mathcal{A}}}\\
{_i\mathbf{f}}_{\mathcal{A}}\ar[r]&\mathbf{f}_{\mathcal{A}}\ar[r]^-{_i\pi_{\mathcal{A}}}&{_i\mathbf{f}}_{\mathcal{A}}
}
$$

Since ${_ij}_{!}:{_i\mathbf{k}}\rightarrow\mathbf{k}$ is a monomorphism of algebras and ${_{i}\mathbf{f}}_{\mathcal{A}}$ is a subalgebra of $\mathbf{f}_{\mathcal{A}}$, the $\mathcal{A}$-algebra isomorphism $\lambda_{\mathcal{A}}:\mathbf{f}_{\mathcal{A}}\rightarrow\mathbf{k}$ induces that ${_{i}\lambda}_{\mathcal{A}}:{_{i}\mathbf{f}}_{\mathcal{A}}\rightarrow{_i\mathbf{k}}$ is also an isomorphism of algebras.

\end{proof}

\subsection{Geometric realization of $^i\mathbf{f}$}\label{subsection:4.4}

\subsubsection{}

Let $\mathfrak{i}$ be a subset of $\mathbf{I}$ satisfying the following conditions:
(1) for any $\mathbf{i}\in\mathfrak{i}$, $\mathbf{i}$ is a source; (2) for any $\mathbf{i},\mathbf{j}\in\mathfrak{i}$, there are no arrows between them.

For any $\mathbf{V}\in\mathcal{C}_{\nu}'$, consider a subvariety ${^{\mathfrak{i}}E}_\mathbf{V}$ of $E_\mathbf{V}$
$$
{^{\mathfrak{i}}E}_{\mathbf{V}}=\{x\in E_{\mathbf{V}}\,\,|\,\,\bigoplus_{h\in H,s(h)=\mathbf{i}}x_{h}:V_\mathbf{i}\rightarrow\bigoplus_{h\in H,s(h)=\mathbf{i}}V_{t(h)} \,\,\,\textrm{is injective for any $\mathbf{i}\in\mathfrak{i}$}\}.
$$
Denote by ${^{\mathfrak{i}}j}_{\mathbf{V}}:{^{\mathfrak{i}}E}_{\mathbf{V}}\rightarrow E_{\mathbf{V}}$ the canonical embedding.

Similarly to the notations in Section \ref{subsection:4.2},
the categories ${^{\mathfrak{i}}\mathcal{P}}_{\mathbf{V}}$ and ${^{\mathfrak{i}}\mathcal{Q}}_{\mathbf{V}}$ can be defined.
Let $K({^{\mathfrak{i}}\mathcal{Q}}_{\mathbf{V}})$ be the Grothendieck group.

The canoniccal embedding ${^{\mathfrak{i}}j}_{\mathbf{V}}:{^{\mathfrak{i}}E}_{\mathbf{V}}\rightarrow E_{\mathbf{V}}$ induces
${^{\mathfrak{i}}j}^{\ast}_{\mathbf{V}}:\mathcal{Q}_{\mathbf{V}}\rightarrow{^{\mathfrak{i}}\mathcal{Q}}_{\mathbf{V}}$ and
${^{\mathfrak{i}}j}_{\mathbf{V}!}:K(^\mathfrak{i}\mathcal{Q}_{\mathbf{V}})\rightarrow K(\mathcal{Q}_{\mathbf{V}})$.
Considering all dimension vectors, we have ${^{\mathfrak{i}}j}^{\ast}:K(\mathcal{Q})\rightarrow K({^{\mathfrak{i}}\mathcal{Q}})$ and
${^{\mathfrak{i}}j}_!:K({^{\mathfrak{i}}\mathcal{Q}})\rightarrow K(\mathcal{Q})$.

\begin{proposition}\label{proposition:4.15}
There exists an isomorphism of algebras ${^\mathfrak{i}\hat{\lambda}}_{\mathcal{A}}:K(^\mathfrak{i}\mathcal{Q})\rightarrow{^\mathfrak{i}\hat{\mathbf{f}}}_{\mathcal{A}}$
such that the following diagram commutes
$$
\xymatrix{
K(^\mathfrak{i}\mathcal{Q})\ar[r]^-{{^{\mathfrak{i}}j}_{!}}\ar[d]^-{{^\mathfrak{i}\hat{\lambda}}_{\mathcal{A}}}&K(\mathcal{Q})\ar[r]^-{{^{\mathfrak{i}}j}^{\ast}}\ar[d]^-{\hat{\lambda}_{\mathcal{A}}}&K(^\mathfrak{i}\mathcal{Q})\ar[d]^-{{^\mathfrak{i}\hat{\lambda}}_{\mathcal{A}}}\\
{^\mathfrak{i}\hat{\mathbf{f}}}_{\mathcal{A}}\ar[r]^-{}&\hat{\mathbf{f}}_{\mathcal{A}}\ar[r]^-{{^\mathfrak{i}\pi}_{\mathcal{A}}}&{^\mathfrak{i}\hat{\mathbf{f}}}_{\mathcal{A}}
}
$$
\end{proposition}

\subsubsection{}

Let $\tilde{Q}=(Q,a)$ be a quiver with automorphism, where $Q=(\mathbf{I},H,s,t)$. Fix $i\in I=\mathbf{I}^a$ and assume that $\mathbf{i}$ is a source for any $\mathbf{i}\in i$.

Similarly to the notations in Section \ref{subsection:4.3}, the category ${^i\tilde{\mathcal{Q}}_{\mathbf{V}}}$ can be defined.
Let $K(^i\tilde{\mathcal{Q}}_{\mathbf{V}})$ be the Grothendieck group. We also have
${^ij}^{\ast}_{\mathbf{V}}:\tilde{\mathcal{Q}}_{\mathbf{V}}\rightarrow{^i\tilde{\mathcal{Q}}_{\mathbf{V}}}$
and
${^{{i}}j}_!:K(^{i}\tilde{\mathcal{Q}}_{\mathbf{V}})\rightarrow K(\tilde{\mathcal{Q}}_{\mathbf{V}})$.

Consider the subalgebra ${^i\mathbf{k}}$ of $K({^i\tilde{\mathcal{Q}}})$.
The functor ${^ij}^{\ast}_{\mathbf{V}}:\tilde{\mathcal{Q}}_{\mathbf{V}}\rightarrow{^i\tilde{\mathcal{Q}}_{\mathbf{V}}}$ induces a map ${^ij}^{\ast}:\mathbf{k}\rightarrow{^i\mathbf{k}}$.
The map ${^{i}j}_!:K(^{i}\tilde{\mathcal{Q}})\rightarrow K(\tilde{\mathcal{Q}})$ induces
a map ${^ij}_{!}:{^i\mathbf{k}}\rightarrow\mathbf{k}$.

\begin{theorem}\label{theorem:4.16}
There is a unique $\mathcal{A}$-algebra isomorphism
$${^i\lambda}_{\mathcal{A}}:{^i\mathbf{k}}\rightarrow{^i\mathbf{f}}_{\mathcal{A}}$$
such that the following diagram commute
$$
\xymatrix{
{^i\mathbf{k}}\ar[d]^-{{^i\lambda}_{\mathcal{A}}}\ar[r]^-{{^ij}_{!}}&\mathbf{k}\ar[r]^-{{^ij}^\ast}\ar[d]^-{\lambda_{\mathcal{A}}}&{^i\mathbf{k}}\ar[d]^-{{^i\lambda}_{\mathcal{A}}}\\
{^i\mathbf{f}}_{\mathcal{A}}\ar[r]&\mathbf{f}_{\mathcal{A}}\ar[r]^-{^i\pi_{\mathcal{A}}}&{^i\mathbf{f}}_{\mathcal{A}}
}
$$
\end{theorem}

\section{Geometric realization of $T_i:{_i\mathbf{f}}\rightarrow{^i\mathbf{f}}$}\label{section:5}

\subsection{Geometric realization}\label{subsection:5.1}

\subsubsection{}

Let $Q=(\mathbf{I},H,s,t)$ be a quiver.
Let $\mathfrak{i}$ be a subset of $\mathbf{I}$ satisfying the following conditions:
(1) for any $\mathbf{i}\in\mathfrak{i}$, $\mathbf{i}$ is a sink; (2) for any $\mathbf{i},\mathbf{j}\in\mathfrak{i}$, there are no arrows between them.

Let $Q'=\sigma_\mathfrak{i}Q=(\mathbf{I},H',s,t)$ be the quiver by reversing the directions
of all arrows in $Q$ containing $\mathbf{i}\in\mathfrak{i}$.
So for any $\mathbf{i}\in\mathfrak{i}$, $\mathbf{i}$ is a source of $Q'$.

For any $\nu,\nu'\in\mathbb{N}\mathbf{I}$ such that $\nu'=s_\mathfrak{i}\nu$ and $\mathbf{V}\in\mathcal{C}'_{\nu}$, $\mathbf{V}'\in\mathcal{C}'_{\nu'}$, consider the following correspondence (\cite{Lusztig_Canonical_bases_and_Hall_algebras,Kato_PBW_bases_and_KLR_algebras})
\begin{equation}\label{equation:5.1.1}
\xymatrix{
_\mathfrak{i}E_{\mathbf{V},Q}&Z_{\mathbf{V}\mathbf{V}'}\ar[l]_-{\alpha}\ar[r]^-{\beta}&^\mathfrak{i}E_{\mathbf{V}',Q'}},
\end{equation}
where
\begin{enumerate}
  \item[(1)]$Z_{\mathbf{V}\mathbf{V}'}$ is the subset of $E_{\mathbf{V},Q}\times E_{\mathbf{V}',Q'}$ consisting of all $(x,y)$ satisfying the following conditions:
           \begin{enumerate}
                  \item[(a)]for any $h\in H$ such that $t(h)\not\in \mathfrak{i}$, $x_h=y_h$;
                  \item[(b)]for any $\mathbf{i}\in\mathfrak{i}$, the following sequence is exact
                  $$
                  \xymatrix@=15pt{0\ar[r]&}
                  \xymatrix@=60pt{V'_\mathbf{i}\ar[r]^-{\bigoplus_{h\in H',s(h)=\mathbf{i}}y_h}&\bigoplus_{h\in H,t(h)=\mathbf{i}}V_{s(h)}\ar[r]^-{\bigoplus_{h\in H,t(h)=\mathbf{i}}x_h}&V_\mathbf{i}}
                  \xymatrix@=15pt{\ar[r]&0};
                  $$
           \end{enumerate}
  \item[(2)]$\alpha(x,y)=x$ and $\beta(x,y)=y$.
\end{enumerate}
From now on, $_\mathfrak{i}E_{\mathbf{V},Q}$ is denoted by $_\mathfrak{i}E_{\mathbf{V}}$ and $^\mathfrak{i}E_{\mathbf{V}',Q'}$ is denoted by $^\mathfrak{i}E_{\mathbf{V}'}$.
Let
\begin{eqnarray*}
G_{\mathbf{V}\mathbf{V}'}&=&\prod_{\mathbf{i}\in\mathfrak{i}}GL(V_\mathbf{i})\times \prod_{\mathbf{i}\in\mathfrak{i}}GL(V'_\mathbf{i})\times\prod_{\mathbf{j}\not\in\mathfrak{i}}GL(V_\mathbf{j})\\
&\cong& \prod_{\mathbf{i}\in\mathfrak{i}}GL(V_\mathbf{i})\times \prod_{\mathbf{i}\in\mathfrak{i}}GL(V'_\mathbf{i})\times\prod_{\mathbf{j}\not\in\mathfrak{i}}GL(V'_\mathbf{j}),\nonumber
\end{eqnarray*}
which acts on $Z_{\mathbf{V}\mathbf{V}'}$ naturally.

By (\ref{equation:5.1.1}), we have
$$
\xymatrix{
\mathcal{D}_{G_{\mathbf{V}}}(_\mathfrak{i}E_{\mathbf{V}})\ar[r]^-{\alpha^{\ast}}&\mathcal{D}_{G_{\mathbf{V}\mathbf{V}'}}(Z_{\mathbf{V}\mathbf{V}'})&\mathcal{D}_{G_{\mathbf{V}'}}(^\mathfrak{i}E_{\mathbf{V}'})\ar[l]_-{\beta^{\ast}}
}.$$
Since $\alpha$ and $\beta$ are principal bundles with fibers $\prod_{\mathbf{i}\in\mathfrak{i}}\mathrm{Aut}(V'_\mathbf{i})$ and $\prod_{\mathbf{i}\in\mathfrak{i}}\mathrm{Aut}(V_\mathbf{i})$ respectively, $\alpha^{\ast}$ and $\beta^{\ast}$ are equivalences of categories (\cite{Bernstein_Lunts_Equivariant_sheaves_and_functors}).

Hence, for any $\mathcal{L}\in\mathcal{D}_{G_{\mathbf{V}}}(_\mathfrak{i}E_{\mathbf{V}})$，
there exists a unique $\mathcal{L}'\in\mathcal{D}_{G_{\mathbf{V}'}}(^\mathfrak{i}E_{\mathbf{V}'})$ such that $\alpha^{\ast}(\mathcal{L})=\beta^{\ast}(\mathcal{L}')$. Define
\begin{eqnarray*}
\tilde{\omega}_\mathfrak{i}:\mathcal{D}_{G_{\mathbf{V}}}(_\mathfrak{i}E_{\mathbf{V}})&\rightarrow&\mathcal{D}_{G_{\mathbf{V}'}}(^\mathfrak{i}E_{\mathbf{V}'})\\
\mathcal{L}&\mapsto&\mathcal{L}'[-s(\mathbf{V})](-\frac{s(\mathbf{V})}{2})
\end{eqnarray*}
where
$
s(\mathbf{V})=\sum_{\mathbf{i}\in\mathfrak{i}}(\dim\textrm{GL}(V_\mathbf{i})-\dim\textrm{GL}(V'_\mathbf{i})).
$
Since $\alpha^{\ast}$ and $\beta^{\ast}$ are equivalences of categories,  $\tilde{\omega}_i$ is also an equivalence of categories.

\begin{proposition}\label{proposition:5.1}
It holds that  $\tilde{\omega}_\mathfrak{i}({_\mathfrak{i}\mathcal{Q}}_{\mathbf{V}})={^\mathfrak{i}\mathcal{Q}}_{\mathbf{V}'}$.
\end{proposition}

For the proof of Proposition \ref{proposition:5.1}, we give some new notations.
Let $\mathfrak{i}_1$ be a subset of $\mathfrak{i}$ and $\mathfrak{i}_2=\mathfrak{i}-\mathfrak{i}_1$, consider the quiver $\sigma_{\mathfrak{i}_1}Q$.
For any $\mathbf{V}\in\mathcal{C}_{\nu}'$, consider a subvariety $^{\mathfrak{i}_1}_{\mathfrak{i}_2}E_{\mathbf{V},\sigma_{\mathfrak{i}_1}Q}$ of $E_{\mathbf{V},\sigma_{\mathfrak{i}_1}Q}$
\begin{eqnarray*}
&&^{\mathfrak{i}_1}_{\mathfrak{i}_2}E_{\mathbf{V},\sigma_{\mathfrak{i}_1}Q}=\{x\in E_{\mathbf{V},\sigma_{\mathfrak{i}_1}Q}\,\,|\,\,\bigoplus_{h\in H,t(h)=\mathbf{i}}x_{h}:\bigoplus_{h\in H,t(h)=\mathbf{i}}V_{s(h)}\rightarrow V_\mathbf{i} \,\,\textrm{is surjective}\\ &&\textrm{for any $\mathbf{i}\in\mathfrak{i}_2$ and} \bigoplus_{h\in H,s(h)=\mathbf{i}}x_{h}: V_\mathbf{i}\rightarrow\bigoplus_{h\in H,s(h)=\mathbf{i}}V_{t(h)} \,\,\,\textrm{is injective for any $\mathbf{i}\in\mathfrak{i}_1$}\}.
\end{eqnarray*}
Denote by $^{\mathfrak{i}_1}_{\mathfrak{i}_2}j_{\mathbf{V},\sigma_{\mathfrak{i}_1}Q}:{^{\mathfrak{i}_1}_{\mathfrak{i}_2}}E_{\mathbf{V},\sigma_{\mathfrak{i}_1}Q}\rightarrow E_{\mathbf{V},\sigma_{\mathfrak{i}_1}Q}$ the canonical embedding.

Let $\mathfrak{i}=\{\mathbf{i}_1,\mathbf{i}_2,\ldots,\mathbf{i}_l\}$, $\mathfrak{i}_k=\{\mathbf{i}_{k},\ldots,\mathbf{i}_{l-1},\mathbf{i}_l\}$ and
$\mathfrak{i}'_k=\mathfrak{i}-\mathfrak{i}_k$.
For any $\nu_1,\nu_2,\ldots,\nu_{l+1} \in\mathbb{N}\mathbf{I}$ such that $\nu_{k+1} =s_{\mathbf{i}_k}\nu_k$ and $\mathbf{V}^k\in\mathcal{C}'_{\nu_k}$, consider the following commutative diagrams
$$
\xymatrix{
_{{\mathfrak{i}_k}}^{\mathfrak{i}'_k}E_{\mathbf{V}^k,Q_k}\ar[d]^-{_{{\mathfrak{i}_k}}^{\mathfrak{i}'_k}j_{\mathbf{V}^k,Q_k}}
&\hat{Z}_{\mathbf{V}^k\mathbf{V}^{k+1}}\ar[l]_-{\alpha_k}\ar[r]^-{\beta_k}\ar[d]^{j_k}
&_{\mathfrak{i}_{k+1}}^{\mathfrak{i}'_{k+1}}E_{\mathbf{V}^{k+1},Q_{k+1}}\ar[d]^-{_{{\mathfrak{i}_{k+1}}}^{\mathfrak{i}'_{k+1}}j_{\mathbf{V}^{k+1},Q_{k+1}}}
\\
_{\mathbf{i}_k}E_{\mathbf{V}^k,Q_k}&Z_{\mathbf{V}^k\mathbf{V}^{k+1}}\ar[l]_-{\alpha_k}\ar[r]^-{\beta_k}&^{\mathbf{i}_k}E_{\mathbf{V}^{k+1},Q_{k+1}}
}$$
where
\begin{enumerate}
  \item[(1)]$\hat{Z}_{\mathbf{V}^k\mathbf{V}^{k+1}}$ is the subset of $_{{\mathfrak{i}_k}}^{\mathfrak{i}'_k}E_{\mathbf{V}^k,Q_k}\times {_{{\mathfrak{i}_{k+1}}}^{\mathfrak{i}'_{k+1}}}E_{\mathbf{V}^{k+1},Q_{k+1}}$ consisting of all $(x,y)$ satisfying the following conditions:
           \begin{enumerate}
                  \item[(a)]for any $h\in H$ such that $t(h)\neq\mathbf{i}_k$, $x_h=y_h$;
                  \item[(b)]the following sequence is exact
                  $$
                  \xymatrix@=15pt{0\ar[r]&}
                  \xymatrix@=60pt{V^{k+1}_{\mathbf{i}_k}\ar[r]^-{\bigoplus_{h\in H',s(h)={\mathbf{i}_k}}y_h}&\bigoplus_{h\in H,t(h)={\mathbf{i}_k}}V^k_{s(h)}\ar[r]^-{\bigoplus_{h\in H,t(h)={\mathbf{i}_k}}x_h}&V^k_{\mathbf{i}_k}}
                  \xymatrix@=15pt{\ar[r]&0};
                  $$
           \end{enumerate}
  \item[(2)]$\alpha_k(x,y)=x$ and $\beta_k(x,y)=y$;
  \item[(3)]$j_k:\hat{Z}_{\mathbf{V}^k\mathbf{V}^{k+1}}\rightarrow Z_{\mathbf{V}^k\mathbf{V}^{k+1}}$  is the canonical embedding.
\end{enumerate}
The algebraic group $G_{\mathbf{V}^k\mathbf{V}^{k+1}}$ also acts on $\hat{Z}_{\mathbf{V}^k\mathbf{V}^{k+1}}$ naturally.
Then we have
$$
\xymatrix{
\mathcal{D}_{G_{\mathbf{V}^k}}(_{\mathbf{i}_k}E_{\mathbf{V}^k,Q_k})\ar[r]^-{\alpha_k^{\ast}}\ar[d]^-{_{{\mathfrak{i}_k}}^{\mathfrak{i}'_k}j^\ast_{\mathbf{V}^k,Q_k}}
&\mathcal{D}_{G_{\mathbf{V}^k\mathbf{V}^{k+1}}}(Z_{\mathbf{V}^k\mathbf{V}^{k+1}})\ar[d]^{j^\ast_k}
&\mathcal{D}_{G_{\mathbf{V}^{k+1}}}({^{\mathbf{i}_k}E}_{\mathbf{V}^{k+1},Q_{k+1}})\ar[l]_-{\beta_k^{\ast}}\ar[d]^-{_{{\mathfrak{i}_{k+1}}}^{\mathfrak{i}'_{k+1}}j^\ast_{\mathbf{V}^{k+1},Q_{k+1}}}\\
\mathcal{D}_{G_{\mathbf{V}^k}}(_{{\mathfrak{i}_k}}^{\mathfrak{i}'_k}E_{\mathbf{V}^k,Q_k})\ar[r]^-{\alpha_k^{\ast}}
&\mathcal{D}_{G_{\mathbf{V}^k\mathbf{V}^{k+1}}}(\hat{Z}_{\mathbf{V}^k\mathbf{V}^{k+1}})
&\mathcal{D}_{G_{\mathbf{V}^{k+1}}}(_{\mathfrak{i}_{k+1}}^{\mathfrak{i}'_{k+1}}E_{\mathbf{V}^{k+1},Q_{k+1}})\ar[l]_-{\beta_k^{\ast}}
}$$
where $\alpha_k^{\ast}$ and $\beta_k^{\ast}$ are equivalences of categories.

Hence, for any $\mathcal{L}\in\mathcal{D}_{G_{\mathbf{V}^k}}(_{{\mathfrak{i}_k}}^{\mathfrak{i}'_k}E_{\mathbf{V}^k,Q_k})$,
there exists a unique $\mathcal{L}'\in\mathcal{D}_{G_{\mathbf{V}^{k+1}}}(_{\mathfrak{i}_{k+1}}^{\mathfrak{i}'_{k+1}}E_{\mathbf{V}^{k+1},Q_{k+1}})$ such that $\alpha_k^{\ast}(\mathcal{L})=\beta_k^{\ast}(\mathcal{L}')$. Define
\begin{eqnarray*}
\hat{\tilde{\omega}}_{\mathbf{i}_k}:\mathcal{D}_{G_{\mathbf{V}^k}}(_{{\mathfrak{i}_k}}^{\mathfrak{i}'_k}E_{\mathbf{V}^k,Q_k})&\rightarrow&
\mathcal{D}_{G_{\mathbf{V}^{k+1}}}(_{\mathfrak{i}_{k+1}}^{\mathfrak{i}'_{k+1}}E_{\mathbf{V}^{k+1},Q_{k+1}})\\
\mathcal{L}&\mapsto&\mathcal{L}'[-s(\mathbf{V}^k)](-\frac{s(\mathbf{V}^k)}{2})
\end{eqnarray*}
where
$
s(\mathbf{V}^k)=\dim\textrm{GL}(V^k_{\mathbf{i}_k})-\dim\textrm{GL}(V^{k+1}_{\mathbf{i}_k}).
$
Since $\alpha_k^{\ast}$ and $\beta_k^{\ast}$ are equivalences of categories, $\hat{\tilde{\omega}}_{\mathbf{i}_k}$ is also an equivalence of categories and we have the following commutative diagram
$$
\xymatrix{
\mathcal{D}_{G_{\mathbf{V}^k}}(_{\mathbf{i}_k}E_{\mathbf{V}^k,Q_k})\ar[r]^-{\tilde{\omega}_{\mathbf{i}_k}}\ar[d]^-{_{{\mathfrak{i}_k}}^{\mathfrak{i}'_k}j^\ast_{\mathbf{V}^k,Q_k}}
&\mathcal{D}_{G_{\mathbf{V}^{k+1}}}(^{\mathbf{i}_k}E_{\mathbf{V}^{k+1},Q_{k+1}})\ar[d]^-{_{{\mathfrak{i}_{k+1}}}^{\mathfrak{i}'_{k+1}}j^\ast_{\mathbf{V}^{k+1},Q_{k+1}}}\\
\mathcal{D}_{G_{\mathbf{V}^k}}(_{{\mathfrak{i}_k}}^{\mathfrak{i}'_k}E_{\mathbf{V}^k,Q_k})\ar[r]^-{\hat{\tilde{\omega}}_{\mathbf{i}_k}}
&\mathcal{D}_{G_{\mathbf{V}^{k+1}}}(_{\mathfrak{i}_{k+1}}^{\mathfrak{i}'_{k+1}}E_{\mathbf{V}^{k+1},Q_{k+1}})
}$$

In \cite{Xiao_Zhao_Geometric_realizations_of_Lusztig's_symmetries}, it was proved that $\tilde{\omega}_{\mathbf{i}_k}(_{\mathbf{i}_k}\mathcal{Q}_{{\mathbf{V}^k,Q_k}})={^{\mathbf{i}_k}}\mathcal{Q}_{\mathbf{V}^{k+1},Q_{k+1}}$.
Hence
\begin{equation}\label{equation:5.1.2}
\hat{\tilde{\omega}}_{\mathbf{i}_k}(_{{\mathfrak{i}_k}}^{\mathfrak{i}'_k}\mathcal{Q}_{{\mathbf{V}^k,Q_k}})
={_{\mathfrak{i}_{k+1}}^{\mathfrak{i}'_{k+1}}}\mathcal{Q}_{\mathbf{V}^{k+1},Q_{k+1}}.
\end{equation}

\begin{proof}[The proof of Proposition \ref{proposition:5.1}]
Note that $\tilde{\omega}_\mathfrak{i}=\prod^{l}_{k=1}\hat{\tilde{\omega}}_{\mathbf{i}_k}$. Formula (\ref{equation:5.1.2}) implies $$\tilde{\omega}_\mathfrak{i}({_\mathfrak{i}\mathcal{Q}}_{\mathbf{V}})={^\mathfrak{i}\mathcal{Q}}_{\mathbf{V}'}.$$

\end{proof}

Hence, we can define
$\tilde{\omega}_\mathfrak{i}:{_\mathfrak{i}\mathcal{Q}}_\mathbf{V}\rightarrow{^\mathfrak{i}\mathcal{Q}}_{\mathbf{V}'}$
and
$\tilde{\omega}_\mathfrak{i}:K({_\mathfrak{i}\mathcal{Q}})\rightarrow K({^\mathfrak{i}\mathcal{Q}})$.

\subsubsection{}

Let $\tilde{Q}=(Q,a)$ be a quiver with automorphism, where $Q=(\mathbf{I},H,s,t)$. Let $i\in I=\mathbf{I}^a$ and $\mathbf{i}$ is a sink for any $\mathbf{i}\in i$.

Consider the following commutative diagram
$$
\xymatrix{
_{i}E_{\mathbf{V}}\ar[d]_-{a}&Z_{\mathbf{V}\mathbf{V}'}\ar[d]_-{a}\ar[r]^-{\beta}\ar[l]_-{\alpha}&^{i}E_{\mathbf{V}'}\ar[d]_-{a}\\
_{i}E_{\mathbf{V}}&Z_{\mathbf{V}\mathbf{V}'}\ar[l]_-{\alpha}\ar[r]^-{\beta}&^{i}E_{\mathbf{V}'}
}
$$
where $a:Z_{\mathbf{V}\mathbf{V}'}\rightarrow Z_{\mathbf{V}\mathbf{V}'}$ is defined by $a(x,y)=(a(x),a(y))$ for any $(x,y)\in Z_{\mathbf{V}\mathbf{V}'}$.
Hence we have $a^\ast\tilde{\omega}_i=\tilde{\omega}_ia^\ast$.
So we have a functor $\tilde{\omega}_i:{_i\tilde{\mathcal{Q}}}_{\mathbf{V}}\rightarrow{^i\tilde{\mathcal{Q}}}_{\mathbf{V}'}$
and a map $\tilde{\omega}_i:K({_i\tilde{\mathcal{Q}}}_{\mathbf{V}})\rightarrow K({^i\tilde{\mathcal{Q}}}_{\mathbf{V}'})$.

\begin{proposition}\label{proposition:5.2}
The map $\tilde{\omega}_i:K({_i\tilde{\mathcal{Q}}}_{\mathbf{V}})\rightarrow K({^i\tilde{\mathcal{Q}}}_{\mathbf{V}'})$ is an isomorphism of algebras.
\end{proposition}

\begin{proof}
Consider the following commutative diagram
$$
\xymatrix{
{_{i}E}_{\mathbf{V}_1}\times {_{i}E}_{\mathbf{V}_2}&{_{i}E}'\ar[l]_-{p_1}\ar[r]^-{p_2}&{_{i}E}''\ar[r]^-{p_3}&{_{i}E}_{\mathbf{V}}\\
Z_{\mathbf{V}_1\mathbf{V}'_1}\times Z_{\mathbf{V}_2\mathbf{V}'_2}\ar[d]^-{\beta}\ar[u]_-{\alpha}&Z'\ar[d]^-{\beta_1}\ar[u]_-{\alpha_1}\ar[l]_-{p_1}\ar[r]^-{p_2}&Z''\ar[d]^-{\beta_2}\ar[u]_-{\alpha_2}\ar[r]^-{p_3}&Z_{\mathbf{V}\mathbf{V}'}\ar[d]^-{\beta}\ar[u]_-{\alpha}\\
{^{i}E}_{\mathbf{V}_1'}\times {^{i}E}_{\mathbf{V}_2'}&{^{i}E}'\ar[l]_-{p_1}\ar[r]^-{p_2}&{^{i}E}''\ar[r]^-{p_3}&{^{i}E}_{\mathbf{V}'}
}
$$
where
\begin{enumerate}
  \item[(1)]$Z''$ is a subset of ${_{i}E}''\times{^{i}E}''$ consisting of the elements $(x,\mathbf{W};y,\mathbf{W}')$ satisfying the following conditions:
           \begin{enumerate}
                  \item[(a)]$(x,y)\in Z_{\mathbf{V}\mathbf{V}'}$;
                  \item[(b)]for any $h\in H$ such that $t(h)\not\in {i}$, $(x|_{\mathbf{W}})_h=(y|_{\mathbf{W}'})_h$;
                  \item[(c)]for any $\mathbf{i}\in{i}$, the following sequence is exact
                  $$
                  \xymatrix@=15pt{0\ar[r]&}
                  \xymatrix@=80pt{W'_\mathbf{i}\ar[r]^-{\bigoplus_{h\in H',s(h)=\mathbf{i}}(y|_{\mathbf{W}'})_h}&\bigoplus_{h\in H,t(h)=\mathbf{i}}W_{s(h)}\ar[r]^-{\bigoplus_{h\in H,t(h)=\mathbf{i}}(x|_{\mathbf{W}})_h}&W_\mathbf{i}}
                  \xymatrix@=15pt{\ar[r]&0};
                  $$
           \end{enumerate}
  \item[(2)]$Z'$ is a subset of ${_{i}E}'\times{^{i}E}'$ consisting of the elements $(x,\mathbf{W},R_2,R_1;y,\mathbf{W}',R'_2,R'_1)$ satisfying the following conditions:
           \begin{enumerate}
                  \item[(a)]$(x,\mathbf{W};y,\mathbf{W}')\in Z''$;
                  \item[(b)]$(x_1,y_1)\in Z_{\mathbf{V}_1\mathbf{V}'_1}$ and $(x_2,y_2)\in Z_{\mathbf{V}_2\mathbf{V}'_2}$, where $(x_1,x_2)=p_1(x,\mathbf{W},R_2,R_1)$ and $(y_1,y_2)=p_1(y,\mathbf{W}',R'_2,R'_1)$;
           \end{enumerate}
  \item[(3)]$p_1(x,\mathbf{W},R'',R',;y,\mathbf{W}',R'_2,R'_1)=((x_1,x_2),(y_1,y_2))$;
  \item[(4)]$p_2(x,\mathbf{W},R'',R',;y,\mathbf{W}',R'_2,R'_1)=(x,\mathbf{W};y,\mathbf{W}')$;
  \item[(5)]$p_3(x,\mathbf{W};y,\mathbf{W}')=(x,y)$.
\end{enumerate}

For any two complexes $\mathcal{L}_1\in\mathcal{D}_{G_{\mathbf{V}'}}({_{i}E}_{\mathbf{V}'})$ and $\mathcal{L}_2\in\mathcal{D}_{G_{\mathbf{V}''}}({_{i}E}_{\mathbf{V}''})$, $\mathcal{L}=\mathcal{L}_1\ast\mathcal{L}_2$ is defined as follows.
Let $\mathcal{L}_3=\mathcal{L}_1\otimes\mathcal{L}_2$ and $\mathcal{L}_4=p_1^{\ast}\mathcal{L}_3$. There exists a complex $\mathcal{L}_5$ on ${_{i}E}'$ such that $p_2^{\ast}(\mathcal{L}_5)=\mathcal{L}_4$. $\mathcal{L}$ is defined as $(p_3)_{!}\mathcal{L}_5$.

Let $\mathcal{L}'_1\in\mathcal{D}_{G_{\mathbf{V}'}}({_{i}E}_{\mathbf{V}'})$ be the unique complex such that $\alpha^\ast\mathcal{L}_1=\beta^\ast\mathcal{L}'_1$ and $\mathcal{L}'_2\in\mathcal{D}_{G_{\mathbf{V}''}}({_{i}E}_{\mathbf{V}''})$  be the unique complex such that $\alpha^\ast\mathcal{L}_2=\beta^\ast\mathcal{L}'_2$. $\mathcal{L}'=\mathcal{L}'_1\ast\mathcal{L}'_2$ is defined as follows.
Let $\mathcal{L}'_3=\mathcal{L}'_1\otimes\mathcal{L}'_2$ and $\mathcal{L}'_4=p_1^{\ast}\mathcal{L}'_3$. There exists a complex $\mathcal{L}'_5$ on ${_{i}E}'$ such that $p_2^{\ast}(\mathcal{L}'_5)=\mathcal{L}'_4$. $\mathcal{L}'$ is defined as $(p_3)_{!}\mathcal{L}'_5$.

Since $\alpha_1^\ast\mathcal{L}_4=\alpha_1^\ast p_1^{\ast}\mathcal{L}_3=p_1^{\ast}\alpha^\ast\mathcal{L}_3$
and $\beta_1^\ast\mathcal{L}'_4=\beta_1^\ast p_1^{\ast}\mathcal{L}'_3=p_1^{\ast}\beta^\ast\mathcal{L}'_3$,
we have $\alpha_1^\ast\mathcal{L}_4=\beta_1^\ast\mathcal{L}'_4$. Since $p^\ast_2$ are equivalences of categories, we have $\alpha_2^\ast\mathcal{L}_5=\beta_2^\ast\mathcal{L}'_5$.
Since
$$
\xymatrix{
{_{i}E}''\ar[r]^-{p_3}&{_{i}E}_{\mathbf{V}}\\
Z''\ar[u]_-{\alpha_2}\ar[r]^-{p_3}&Z_{\mathbf{V}\mathbf{V}'}\ar[u]_-{\alpha}
}
$$
and
$$
\xymatrix{
Z''\ar[d]^-{\beta_2}\ar[r]^-{p_3}&Z_{\mathbf{V}\mathbf{V}'}\ar[d]^-{\beta}\\
{^{i}E}''\ar[r]^-{p_3}&{^{i}E}_{\mathbf{V}'}
}
$$
are fiber products, $\alpha^\ast\mathcal{L}=\alpha^\ast(p_3)_{!}\mathcal{L}_5=(p_3)_{!}\alpha_2^\ast\mathcal{L}_5
=(p_3)_{!}\beta_2^\ast\mathcal{L}'_5=\beta^\ast(p_3)_{!}\mathcal{L}'_5=\beta^\ast\mathcal{L}'$.

Hence we have $\tilde{\omega}_i(\mathcal{L}_1\circledast\mathcal{L}_2)=\tilde{\omega}_i\mathcal{L}_1\circledast\tilde{\omega}_i\mathcal{L}_2$ and
the map $\tilde{\omega}_i:K({_i\tilde{\mathcal{Q}}}_{\mathbf{V}})\rightarrow K({^i\tilde{\mathcal{Q}}}_{\mathbf{V}'})$ is an isomorphism of algebras.

\end{proof}

\begin{proposition}\label{proposition:5.3}
For all $i\neq j\in I$ and $-a_{ij}\geq m\in\mathbb{N}$,
$\tilde{\omega}_i({_i\lambda}^{-1}_{\mathcal{A}}(f(i,j,m)))={^i\lambda}^{-1}_{\mathcal{A}}(f'(i,j,-a_{ij}-m)).$
\end{proposition}

Proposition \ref{proposition:5.3} will be proved in Section \ref{subsection:5.2}.

\begin{theorem}\label{theorem:5.4}
It holds that $\tilde{\omega}_i({_i\mathbf{k}})={{^i\mathbf{k}}}$
and we have the following commutative diagram
$$\xymatrix{
{_i\mathbf{k}}\ar[r]^-{\tilde{\omega}_i}\ar[d]^-{{_i\lambda}_{\mathcal{A}}}&{{^i\mathbf{k}}}\ar[d]^-{{^i\lambda}_{\mathcal{A}}}\\
{_i}\mathbf{f}_{\mathcal{A}}\ar[r]^-{T_i}&{^i\mathbf{f}}_\mathcal{A}
}$$
\end{theorem}

\begin{proof}
Since ${_i\mathbf{k}}$ and ${^i\mathbf{k}}$ are generated by ${_i\lambda}^{-1}_{\mathcal{A}}(f(i,j,m))$ and ${^i\lambda}^{-1}_{\mathcal{A}}(f'(i,j,m'))$ respectively, Proposition \ref{proposition:5.3} implies this theorem.
\end{proof}


\subsection{The proof of Proposition \ref{proposition:5.3}}\label{subsection:5.2}

\subsubsection{}

Let $\tilde{Q}=(Q,a)$ be a quiver with automorphism, where $Q=(\mathbf{I},H,s,t)$. Fix $i,j\in I$ such that there are no arrows from $\mathbf{i}$ to $\mathbf{j}$ for any $\mathbf{i}\in i$ and $\mathbf{j}\in j$.
Let $N=|\{\mathbf{j}\rightarrow \mathbf{i}\,\,|\,\,\mathbf{j}\in j, \mathbf{i} \in i\}|$ and $m$ be a non-negative integer such that $m\leq N$.
Let $\gamma_i=\sum_{\mathbf{i}\in i}\mathbf{i}$, $\gamma_j=\sum_{\mathbf{j}\in j}\mathbf{j}$
and $\nu^{(m)}=m\gamma_i+\gamma_j\in\mathbb{N}\mathbf{I}$. Fix an object $\mathbf{V}^{(m)}\in\tilde{\mathcal{C}}$ such that $\underline{\dim}\mathbf{V}^{(m)}=\nu^{(m)}$.

Denote by $\mathbf{1}_{_iE_{\mathbf{V}^{(m)}}}\in\mathcal{D}_{G_{\mathbf{V}^{(m)}}}(_iE_{\mathbf{V}^{(m)}})$ the constant sheaf on $_iE_{\mathbf{V}^{(m)}}$.
Define
$$\mathcal{E}^{(m)}=j_{\mathbf{V}^{(m)}!}(v^{-mN}\mathbf{1}_{_iE_{\mathbf{V}^{(m)}}})\in\mathcal{D}_{G_{\mathbf{V}^{(m)}}}(E_{\mathbf{V}^{(m)}})$$
For convenience, the complex $j_{\mathbf{V}^{(m)}!}(\mathbf{1}_{_iE_{\mathbf{V}^{(m)}}})\in\mathcal{D}_{G_{\mathbf{V}^{(m)}}}(E_{\mathbf{V}^{(m)}})$ is also denoted by
$\mathbf{1}_{_iE_{\mathbf{V}^{(m)}}}$.
Note that there exists a canonical isomorphism
$\mathrm{id}:a^{\ast}(\mathcal{E}^{(m)})\cong\mathcal{E}^{(m)}$.

For each $m\geq p\in\mathbb{N}$, consider the following variety
\begin{eqnarray*}
\tilde{S}^{(m)}_p&=&\{(x,W)\,\,|\,\,x\in E_{\mathbf{V}^{(m)}},\,W=\oplus_{\mathbf{i}\in i}W_{\mathbf{i}}\subset\oplus_{\mathbf{i}\in i}V_\mathbf{i},\,\\
&&a(W)=W,\,\dim(W_{\mathbf{i}})=p,\,\textrm{Im}\bigoplus_{h\in H,t(h)=\mathbf{i}}x_{h}\subset W_{\mathbf{i}}\}.
\end{eqnarray*}
Let $\pi_p:\tilde{S}^{(m)}_p\rightarrow E_{\mathbf{V}^{(m)}}$ be the projection taking $(x,W)$ to $x$ and $S^{(m)}_p=\textrm{Im}\pi_p$.

By the definitions of $S^{(m)}_p$, we have
$${E_{\mathbf{V}^{(m)}}}=S^{(m)}_m\supset{S^{(m)}_{m-1}}\supset{S^{(m)}_{m-2}}\supset\cdots\supset{S^{(m)}_0}.$$
For each $1\leq p\leq m$, let $$\mathcal{N}^{(m)}_{p}={S^{(m)}_{p}}\backslash{S^{(m)}_{p-1}}.$$
Denote by $i^{(m)}_p:{S^{(m)}_{p-1}}\rightarrow{S^{(m)}_{p}}$ the close embedding and $j^{(m)}_p:\mathcal{N}^{(m)}_{p}\rightarrow{S^{(m)}_{p}}$ the open embedding.

Define
$$I^{(m)}_p={(\pi_p)}_!(\mathbf{1}_{\tilde{S}^{(m)}_p})[\dim\tilde{S}^{(m)}_p].$$
Since $a^{\ast}(\mathbf{1}_{\tilde{S}^{(m)}_p})\cong\mathbf{1}_{\tilde{S}^{(m)}_p}$, we have an isomorphism $\phi_0:a^{\ast}(I^{(m)}_p)\cong I^{(m)}_p$.

The following theorem is the main result in this section.

\begin{theorem}\label{theorem:5.5}
For $\mathcal{E}^{(m)}$, there exists $s_{m}\in\mathbb{N}$. For each $s_m\geq p\in\mathbb{N}$, there exists $\mathcal{E}_p^{(m)}\in\mathcal{D}_{G_{\mathbf{V}^{(m)}}}(E_{\mathbf{V}^{(m)}})$ such that
\begin{enumerate}
\item[(1)]$\mathcal{E}_{s_m}^{(m)}=\mathcal{E}^{(m)}$ and $\mathcal{E}_{0}^{(m)}$ is the direct sum of some semisimple perverse sheaves of the form $I^{(m)}_{p'}[l]$;
\item[(2)]for each $p\geq1$, there exists a  distinguished triangle
   \begin{displaymath}
       \xymatrix{
      \mathcal{E}^{(m)}_{p}\ar[r]&{\mathcal{G}^{(m)}_{p}}\ar[r]&\mathcal{E}^{(m)}_{p-1}\ar[r]&,
      }
      \end{displaymath}
      where $\mathcal{G}^{(m)}_{p}$ is the direct sum of some semisimple perverse sheaves of the form $I^{(m)}_{p'}[l]$.
\end{enumerate}
\end{theorem}


The proof of Theorem \ref{theorem:5.5} is as same as that of Theorem 5.3 in \cite{Xiao_Zhao_Geometric_realizations_of_Lusztig's_symmetries}.

\begin{corollary}\label{corollary:5.6}
For each $N\geq m\in\mathbb{N}$, we have the following formula
$$\lambda_{\mathcal{A}}([\mathcal{E}^{(m)},\mathrm{id}])=\sum_{p=0}^{m}(-1)^p{v_i}^{-p(1+N-m)}\theta_i^{(p)}\theta_j\theta_i^{(m-p)}=f(i,j;m).$$
\end{corollary}

\begin{proof}
Since $\lambda_{\mathcal{A}}([I^{(m)}_p,\phi_0])=\theta_i^{(m-p)}\theta_j\theta_i^{(p)}$ for each $m\geq p\in\mathbb{N}$, we have the desired result.

\end{proof}

\subsubsection{}


Let $m$ be a non-negative integer such that $m\leq N$ and $m'=N-m$. Let $\nu=m\gamma_i+\gamma_j\in\mathbb{N}\mathbf{I}$ and $\nu'=s_i\nu=m'\gamma_i+\gamma_j\in\mathbb{N}\mathbf{I}$. Fix two objects $\mathbf{V}\in\tilde{\mathcal{C}}_{\nu}$ and $\mathbf{V}'\in\tilde{\mathcal{C}}_{\nu'}$ .

Denote by $\mathbf{1}_{_iE_{\mathbf{V}}}\in\mathcal{D}_{G_{\mathbf{V}}}(_iE_{\mathbf{V}})$ the constant sheaf on $_iE_{\mathbf{V}}$ and $\mathbf{1}_{^iE_{\mathbf{V}'}}\in\mathcal{D}_{G_{\mathbf{V}'}}(^iE_{\mathbf{V}'})$ the constant sheaf on $^iE_{\mathbf{V}'}$.
Note that there exist canonical isomorphisms
$\mathrm{id}:a^{\ast}(\mathbf{1}_{_iE_{\mathbf{V}}})\cong\mathbf{1}_{_iE_{\mathbf{V}}}$ and
$\mathrm{id}:a^{\ast}(\mathbf{1}_{^iE_{\mathbf{V}'}})\cong\mathbf{1}_{^iE_{\mathbf{V}'}}$.

\begin{proposition}\label{proposition:5.7}
For any $N\geq m\in\mathbb{N}$,
$\tilde{\omega}_i([{v_i}^{-mN}\mathbf{1}_{_iE_{\mathbf{V}}},\mathrm{id}])=[{v_i}^{-m'N}\mathbf{1}_{^iE_{\mathbf{V}'}},\mathrm{id}]$.
\end{proposition}

\begin{proof}[\bf{Proof}]
By the definitions of $\alpha$ and $\beta$ in the diagram (\ref{equation:5.1.1}),
$$\alpha^{\ast}(\mathbf{1}_{_iE_{\mathbf{V}}})=\mathbf{1}_{Z_{\mathbf{V}\mathbf{V}'}}=\beta^{\ast}(\mathbf{1}_{^iE_{\mathbf{V}'}}).$$
Hence $$\tilde{\omega}_i(\mathbf{1}_{_iE_{\mathbf{V}}})={v_i}^{(m-m')N}\mathbf{1}_{^iE_{\mathbf{V}'}}.$$
That is
$$\tilde{\omega}_i({v_i}^{-mN}\mathbf{1}_{_iE_{\mathbf{V}}})={v_i}^{-m'N}\mathbf{1}_{^iE_{\mathbf{V}'}}.$$

\end{proof}



Corollary \ref{corollary:5.6} implies
$${_i\lambda}_{\mathcal{A}}[{v_i}^{-mN}\mathbf{1}_{_iE_{\mathbf{V}}},\mathrm{id}]=f(i,j;m).$$
Similarly, we have $${^i\lambda}_{\mathcal{A}}[{v_i}^{-m'N}\mathbf{1}_{^iE_{\mathbf{V}'}},\mathrm{id}]=f'(i,j;m').$$
Hence
Proposition \ref{proposition:5.7} implies Proposition \ref{proposition:5.3}.

\bibliography{mybibfile}

\begin{thebibliography}{10}

\bibitem{Beilinson_Bernstein_Deligne_Faisceaux_pervers}
A.~Beilinson, J.~Bernstein, and P.~Deligne.
\newblock Faisceaux pervers.
\newblock {\em Ast{\'e}risque}, 100, 1982.

\bibitem{Bernstein_Lunts_Equivariant_sheaves_and_functors}
J.~Bernstein and V.~Lunts.
\newblock {\em Equivariant sheaves and functors}.
\newblock Springer, 1994.

\bibitem{Deng_Xiao}
B.~Deng and J.~Xiao.
\newblock Ringel-{H}all algebras and {L}usztig's symmetries.
\newblock {\em J. Algebra}, 255(2):357--372, 2002.

\bibitem{Kato_An_algebraic_study_of_extension_algebra}
S.~Kato.
\newblock An algebraic study of extension algebras.
\newblock {\em arXiv preprint arXiv:1207.4640}, 2012.

\bibitem{Kato_PBW_bases_and_KLR_algebras}
S.~Kato.
\newblock Poincar{\'e}-{B}irkhoff-{W}itt bases and
  {K}hovanov-{L}auda-{R}ouquier algebras.
\newblock {\em Duke Math. J.}, 163(3):619--663, 2014.

\bibitem{Kiehl_Weissauer_Weil_conjectures_perverse_sheaves_and_l'adic_Fourier_transform}
R.~Kiehl and R.~Weissauer.
\newblock {\em Weil conjectures, perverse sheaves and l'adic {F}ourier
  transform}, volume~42.
\newblock Springer, 2001.

\bibitem{Lusztig_Quantum_deformations_of_certain_simple_modules_over_enveloping_algebras}
G.~Lusztig.
\newblock Quantum deformations of certain simple modules over enveloping
  algebras.
\newblock {\em Adv. Math.}, 70(2):237--249, 1988.

\bibitem{Lusztig_Canonical_bases_arising_from_quantized_enveloping_algebra}
G.~Lusztig.
\newblock Canonical bases arising from quantized enveloping algebras.
\newblock {\em J. Amer. Math. Soc.}, pages 447--498, 1990.

\bibitem{Lusztig_Quantum_groups_at_roots_of_1}
G.~Lusztig.
\newblock Quantum groups at roots of 1.
\newblock {\em Geom. Dedicata}, 35(1):89--113, 1990.

\bibitem{Lusztig_Quivers_perverse_sheaves_and_the_quantized_enveloping_algebras}
G.~Lusztig.
\newblock Quivers, perverse sheaves, and quantized enveloping algebras.
\newblock {\em J. Amer. Math. Soc.}, pages 365--421, 1991.

\bibitem{Lusztig_Braid_group_action_and_canonical_bases}
G.~Lusztig.
\newblock Braid group action and canonical bases.
\newblock {\em Adv. Math.}, 122(2):237--261, 1996.

\bibitem{Lusztig_Canonical_bases_and_Hall_algebras}
G.~Lusztig.
\newblock Canonical bases and {H}all algebras.
\newblock In {\em Representation theories and algebraic geometry}, pages
  365--399. Springer, 1998.

\bibitem{Lusztig_Introduction_to_quantum_groups}
G.~Lusztig.
\newblock {\em Introduction to quantum groups}.
\newblock Springer, 2010.

\bibitem{Ringel_Hall_algebras_and_quantum_groups}
C.~M. Ringel.
\newblock Hall algebras and quantum groups.
\newblock {\em Invent. Math.}, 101(1):583--591, 1990.

\bibitem{Ringel_PBW-bases_of_quantum_groups}
C.~M. Ringel.
\newblock P{BW}-bases of quantum groups.
\newblock {\em J. Reine Angew. Math.}, 470:51--88, 1996.

\bibitem{Sevenhant_Van_den_Bergh_On_the_double_of_the_Hall_algebra_of_a_quiver}
B.~Sevenhant and M.~Van~den Bergh.
\newblock On the double of the {H}all algebra of a quiver.
\newblock {\em J. Algebra}, 221(1):135--160, 1999.

\bibitem{Xiao_Yang_BGP-reflection_functors_and_Lusztig's_symmetries}
J.~Xiao and S.~Yang.
\newblock {BGP}-reflection functors and {L}usztig's symmetries: {A}
  {R}ingel-{H}all algebra approach to quantum groups.
\newblock {\em J. Algebra}, 241(1):204--246, 2001.

\bibitem{Xiao_Zhao_BGP-reflection_functors_and_Lusztig's_symmetries_of_modified_quantized_enveloping_algebras}
J.~Xiao and M.~Zhao.
\newblock {BGP}-reflection functors and {L}usztig's symmetries of modified
  quantized enveloping algebras.
\newblock {\em Acta Math. Sin.}, 29(10):1833--1856, 2013.

\bibitem{Xiao_Zhao_Geometric_realizations_of_Lusztig's_symmetries}
J.~Xiao and M.~Zhao.
\newblock Geometric realizations of lusztig's symmetries.
\newblock {\em arXiv preprint arXiv:1501.01778}, 2015.

\end{thebibliography}

\end{document}